\documentclass[11pt]{article}

\usepackage{amsthm,amsfonts,amssymb,amsmath,color}
\usepackage{amsbsy}
\usepackage{wasysym,verbatim}
\usepackage[rose,frak,operator]{paper_diening}
\newtheorem{defi}{Definition}[section]
\newtheorem{Theorem}[defi]{Theorem}
\newtheorem{Proposition}[defi]{Proposition}
\newtheorem{Lemma}[defi]{Lemma}
\newtheorem{Corollary}[defi]{Corollary}
\newtheorem{Remark}[defi]{Remark}

\newcommand{\nn}{\nonumber}

\newcommand{\R}{{\mathbb R}}

\newcommand{\NN}{{\mathbb N}}

\newcommand{\DC}{C^\infty_c}

\newcommand{\vr}{\varrho}

\newcommand{\tvr}{\tilde \vr}

\newcommand{\vu}{\vc{u}}
\newcommand{\vc}[1]{{\bf #1}}

\newcommand{\Div}{{\rm div}_x}
\newcommand{\Grad}{\nabla_x}

\newcommand{\dx}{{\rm d} {x}}
\newcommand{\dt}{{\rm d} t }

\newcommand{\intO}[1]{\int_{\Omega} #1 \ \dx}
\newcommand{\intOe}[1]{\int_{\Omega_\ep} #1 \ \dx}

\newcommand{\intRt}[1]{\int_{R^3} #1 \ \dx}

\newcommand{\ep}{\varepsilon}

\def\e{\varepsilon}
\def\d{\partial}
\def\O{\Omega}
\def\th{\theta}
\def\de{\delta}
\def\a{\alpha}
\def\g{\gamma}
\def\b{\beta}
\providecommand{\ep}{\e}

\renewcommand\be{\begin{equation}}
\newcommand\ee{\end{equation}}
\renewcommand\ba{\begin{equation}\begin{aligned}}
\newcommand\ea{\end{aligned}\end{equation}}

\renewcommand{\ggg}{{\bf g}}
 \providecommand{\skp}[2]{\skptmp{}{#1}{#2}}
  \providecommand{\norm}[1]{\normtmp{}{#1}}
\definecolor{Cgrey}{rgb}{0.85,0.85,0.85}
\definecolor{Cblue}{rgb}{0.50,0.85,0.85}
\definecolor{Cred}{rgb}{1,0,0}
\definecolor{fancy}{rgb}{0.10,0.85,0.10}

\usepackage{layout}

\numberwithin{equation}{section}

\date{}

\begin{document}


\title{Homogenization of the compressible Navier-Stokes equations in domains with very tiny holes}

\author{Yong Lu\footnote{Chern Institute of Mathematics $\&$ LPMC, Nankai University, Tianjin 300071, China. Email: {\tt lvyong@amss.ac.cn}.} \and Sebastian Schwarzacher\footnote{Department of mathematical analysis, Faculty of Mathematics and Physics, Charles University, Sokolovsk\'a 83, 186 75 Praha, Czech Republic. Email {\tt schwarz@karlin.mff.cuni.cz}.}}

\maketitle

\begin{abstract}

We consider the homogenization problem of the compressible Navier-Stokes equations in a bounded three dimensional domain perforated with very tiny holes.  As the number of holes increases to infinity, we show that, if the size of the holes is small enough, the homogenized equations are the same as the compressible Navier-Stokes equations in the homogeneous domain---domain without holes. This coincides with the previous studies for the Stokes equations and the stationary Navier-Stokes equations. It is the first result of this kind in the instationary barotropic compressible setting. The main technical novelty is the study of the Bogovski{\u\i} operator in non-Lipschitz domains.

\end{abstract}



\section{Introduction}\label{i}

In practice, there comes up the study of fluid flows in domains distributed with a large number of \emph{holes} that represent solid obstacles. It is suitable, for instance to model poluted underground water or oil development. The fluid flows passes between the small obstacles or in the holes in between. Such domains are usually called perforated domains and a typical example is the so-called porous media. The perforation parameters, which are mainly the size of  holes and the mutual distance of the holes in the perforated domain under consideration play a determinant role in these problems. We refer to \cite{book-hom} for a number of real world applications.

Homogenization problems in fluid mechanics represent the study of the asymptotic behavior of fluid flows in perforated domains as the number of holes (obstacles) goes to infinity and the size of holes (obstacles) goes to zero simultaneously. The mathematical concern is the asymptotic behavior of the solutions to equations describing fluid flows with respect to perforation parameters. With an increasing number of holes within the domain of the fluid, the fluid flow approaches an effective state governed by certain "homogenized" equations which in the full domain, i.e. a flow without obstacles.

With different physical backgrounds, the mathematical equations governing the fluid flows vary a lot. There are typically Stokes equations, Navier--Stokes equations or Euler equations. There also include equations describing non--Newtonian fluid flows, such as $p$--Stokes equations, Oldroyd--B models, and many others. Accordingly, the mathematical study of homogenization problems in fluid mechanics includes the homogenization for various fluid models. We refer to \cite{Tartar1, ALL-NS1, ALL-NS2} for Stokes equations, \cite{Mik} for incompressible Navier--Stokes equations, \cite{Mas-Hom} for compressible Navier--Stokes equations and \cite{FNT-Hom} for the complete compressible Navier--Stokes--Fourier equations. We also refer to the book \cite{book-hom2} for other models, such as two phase models and non--Newtonian fluid models.

 In this paper, we study the \emph{homogenization} problem for the compressible Navier-Stokes equations in perforated domains. We consider a bounded three dimensional domain perforated with tiny holes where the diameters of the holes are taken to be of size $O(\e^{\alpha})$ with $\alpha\geq 1$ and with assumed minimal mutual distance between the holes of size $O(\e)$. In this work we will provide a size parameter $\alpha_0\geq 1$ to be precisely specified later on in Theorem \ref{Tm1}, such that for $\alpha\geq \alpha_0$ a homogenization sequence of solutions to the perforated domain converges to the compressible Navier Stokes systems without wholes. This means that if the size of the wholes is small enough than in the homogenization limit they can not be seen anymore. This is a known phenomenon and analogous results for incompressible fluids as well as for the stationary Navier Stokes equation have been proved and will be discussed below in more detail. However, we wish to emphasize that up to our knowledge this result is the first one of its kind for instationary compressible fluids.

\medskip

 Let us introduce the setting in more detail. Let $\Omega \subset \R^3$ be a $C^2$ bounded domain and $\{T_{\e,k}\}_{k\in K_\e}\subset \Omega$ be a family of closed sets (named \emph{holes} or \emph{solid obstacles}) satisfying
 \be\label{dis-holes}
T_{\e,k} = x_{\e,k} + \e^\a T_k \subset B(x_{\e,k}, \de_0 \e^\alpha )\subset B(x_{\e,k}, \de_1 \e^\alpha ) \subset B(x_{\e,k}, \delta_2 \e) \subset B(x_{\e,k}, \de_3\e)\subset \O,
\ee
where for each $k$,  $T_{k}\subset \R^3$ is a simply connected bounded domain of class $C^2$, where the $C^2$ constant is assumed to be uniformly bounded in $k$. With $B(x,r)$ we denote the open ball centered at $x$ with radius $r$ in $\R^3$. Here we assume that $\de_0, \ \de_1, \ \de_2,\ \de_3$ are positive fixed constants independent of $\e$, such that $\de_0<\de_1$ and $\de_2<\de_3$.  Moreover, we suppose that the balls (control volumes) $\{B(x_{\e,n}, \de_3\e)\}_{n\in \NN}$ are pairwise disjoint. The corresponding $\e$-dependent perforated domain is defined as
\be\label{domain}
\Omega_\e:=\Omega \setminus  \bigcup_{k\in K_\e} T_{\e,k}. 
\ee
By the distribution of holes assumed above, the number of holes in $\Omega_\e$ satisfy
\be\label{number-holes}
|K_{\e}|\leq C\frac{|\Omega|}{\e^3},\quad \mbox{for some $C$ independent of $\e$}.
\ee

\medskip

We consider the following Navier--Stokes equations in the space-time cylinder $(0,T)\times \Omega_\e$:
\be\label{i1}
\partial_t \vr + \Div (\vr \vu) = 0,
\ee
\be\label{i2}
\partial_t (\vr \vu) + \Div (\vr \vu \otimes \vu) +\Grad p(\vr) =
\Div {\mathbb S}(\Grad \vu)+\vr {\bf f},
\ee
\be\label{i3}
{\mathbb S} (\Grad \vu) = \mu \left( \Grad \vu + \Grad^t \vu - \frac{2}{3} \Div \vu {\mathbb I} \right) + \eta \Div \vu {\mathbb I},\ \mu > 0,\ \eta \geq 0.
\ee


Here, $\vr$ is the fluid mass density, $\vu $ is the velocity field, $p = p(\vr)$ denotes the pressure,  ${\mathbb S} = {\mathbb S} (\Grad \vu)$  stands for the Newtonian viscous stress tensor with $\mu$, $\lambda$ are the viscosity coefficients, $\vc{f}$ is the external force functions satisfying  $\|\vc{f}\|_{L^\infty((0,T)\times \O;\R^3)}<\infty.$ We remark that this condition on the external force $\ff$ is not optimal and can be improved. However, this improvement does not influence the proof in an essential manner, so we keep this restriction.

We impose the no-slip boundary condition
\be\label{i4}
u=0\quad \mbox{on}\ (0,T)\times \partial\Omega_\e
\ee
and impose the standard technical hypothesis imposed on the pressure in order to ensure the existence of global-in-time weak solutions
to the primitive Navier--Stokes system:
\be\label{pp1}
p \in C[0,\infty) \cap C^3(0, \infty), \ p(0) = 0, \ p'(\vr) > 0 \ \mbox{for}\ \vr > 0,\
\lim_{\vr \to \infty} \frac{p'(\vr)}{\vr^{\gamma - 1}} = p_\infty > 0
\ee
with $\gamma\geq 1$ for which the value range that we can handle is given precisely later on in Theorem \ref{Tm1}.

\medskip

In the sequel,  for a function $f$ defined in $\Omega_\e$, the notation $\tilde f$ or $E(f)$ stands for the zero-extension of $f$ in $\R^3$:
\be\label{extension}
\tilde f=f \quad \mbox{in} \ \Omega_\e,\qquad \tilde f=0 \quad \mbox{in} \ \R^3 \setminus \Omega_\e.\nn
\ee

\subsection{Known results}

We introduce some known results concerning the homogenization problems in the framework of fluid mechanics.

For the case $\a=1$, meaning that the size of holes is proportional to their mutual distance, Tartar \cite{Tartar1} recovered the Darcy law from the homogenization of Stokes equations.

In \cite{ALL-NS1,ALL-NS2}, Allaire gave a systematic study for the homogenization problems of Stokes and stationary incompressible Navier-Stokes equations and showed that the homogenization process crucially depends on the size of the holes. Specifically, in three dimensions, Allaire showed that when $\alpha<3$ corresponding to large holes, the behavior of the limit fluid is governed by the classical Darcy's law, as shown in \cite{Tartar1} for the case $\a=1$; when $\alpha>3$ corresponding to small holes, the equations do not change in the homogenization process and the limit homogenized system is the same system as the orginal Stokes or Navier-Stokes equations in \emph{homogeneous domain}---domain without holes; when $\alpha=3$ corresponding to the critical size of holes, in the limit it yields Brinkman's law which is a combination of Darcy's law and the original equations.

Later on, in the case $\alpha=1$,  the results have been extended to the incompressible instationary Navier-Stokes equations by  Mikeli\'{c} \cite{Mik},  to the compressible Navier--Stokes system by Masmoudi \cite{Mas-Hom}, and to the complete Navier--Stokes--Fourier system in \cite{FNT-Hom}. In all the aforementioned cases, the homogenization limit gives rise to Darcy's law.

Recently, the case $\alpha>3$ was considered in \cite{FL1, DFL} for the stationary compressible Navier-Stokes equation. If in addition $\gamma > 2$ satisfies $\frac{\a(\g-2)}{2\g-3} >1$, it was shown that the limit homogenized equations remain unchanged in homogenized domains. This coincides with the results obtained by Allaire \cite{ALL-NS1,ALL-NS2} for Stokes equations.

Up to our knowledge, the result in this paper is the first analytical result in the study of homogenization for the instationary compressible Navier-Stokes equations when the perforation parameter $\a\neq 1$.

\subsection{Difficulties}

In this section, we illustrate some difficulties in the study of homogenization problems for compressible Navier--Stokes equations. In particular, we point out the main new difficulties compared to the previous study \cite{FL1,DFL} concerning stationary compressible Navier--Stokes equations.

\subsubsection{A review of difficulties in the compressible fluid setting}

In our setting with very small holes, the result is formally similar as for incompressible equations obtained by Allaire \cite{ALL-NS1,ALL-NS2}, where the limit equations remain unchanged. However, the techniques for compressible system are rather different. The main difference is the presence of the pressure term which is a nonlinear function of the density. To obtain the same pressure form in the limit, one needs to show strong convergence of the density. While in the incompressible case, the equations are usually treated in divergence free spaces (testing by functions that are divergence free or applying the Helmholtz decomposition operator), and the pressure is reconstructed through the equation by using Ne\u{c}as's Theorem on negative Sobolev norms.


For compressible Navier-Stokes equation, a priori we only have $L^\infty(0,T;L^1(\Omega_\e))$ uniform estimate for the pressure $p(\vr)$, that is only $L^1$ bounds with respect to spatial variables. {However,} a uniform bounded family in $L^1$ is not weakly pre-compact in $L^1$. It was observed in the study of compressible Navier-Stokes equation in \cite{LI4,F-book,FNP} that one can improve the integrability by employing some kind of divergence inverse operator. For equations in bounded domains, the so-called Bogovski{\u\i} operator, which is the famous inverse operator of $\Div$ to trace free functions~\cite{Bog79,Bog80}, was applied in \cite{FNP} to obtain higher integrability of $p(\vr)$. However, the operator norm of the classical Bogovski{\u\i} operator depends on the Lipschitz character of the domain which is perforated in our case. We refer to \cite[Chapter 3]{Galdi-book} for a construction, as well as a norm estimate, of  the Bogovski{\u\i} operator in Lipschitz domains and star-shaped domains.

It can be found that the Lipschitz norm of the perforated domain $\Omega_\e$ is at least of order $1/\e$ which is unbounded as $\e \to 0$. Thus by employing the classical Bogovski{\u\i} operator one cannot get uniform higher integrability of $p(\vr)$.

 However, the established existence theory for the existence of finite energy weak solutions to the compressible Navier--Stokes equations is well developed for the case $\g>3/2$ and any fixed $\e$. To be able to pass to the limit with $\e\to 0$, we have to establish uniform pressure estimates despite the fact that we do not have domains with uniform Lipschitz boundary.


Since the higher integrability of the pressure is the key  step to prove convergence with $\e\to 0$, we introduce the respective stationary technique of \cite{FL1}~in more detail as an introduction to the problems that we have to handle in the instationary case. The method is  based on finding a convenient inverse of the $\Div$ operator for non Lipschitz domains.
In \cite{FL1} it was considered the case $\g\geq 3$ which guarantees the $L^2$ integrability of the pressure. Indeed, there holds
 \be\label{int-vr}
\vr\in L^{3(\g-1)}(\O_\e), \ \mbox{if} \ 3/2<\g\leq 3; \quad \vr\in L^{2\g}(\O_\e), \ \mbox{if} \ \g\geq 3.
\ee

The construction of an inverse of $\Div$ is done by employing an {\em Restriction operator} $R_\e:\, W_{0}^{1,2}(\Omega;\R^3) \to  W_{0}^{1,2}(\Omega_\e;\R^3)$, that was introduced by Allaire \cite{ALL-NS1} ({see also Section 2.2}). For any $f\in L^2_0(\Omega_\e)$ which is the set of $L^2(\Omega_\e)$ functions with zero mean value, the authors first considered its zero-extension $E(f):=\tilde f$ in $\Omega$ and then employ the Bogovski{\u\i} operator $\mathcal{B}:L^2_0(\Omega)\to W^{1,2}_0(\Omega)$\footnote{{see \cite[Appendix]{F-N-book} or Section 2 below for more details and the definition of the Bogovski{\u\i} operator.}} on the unperforated domain $\Omega$ and find that the operator defined as
\be\label{def-B-in}
\mathcal{B}_\e:=R_\e \circ \mathcal{B} \circ E\ : \ L^2_0(\Omega_\e) \to W_{0}^{1,2}(\Omega_\e;\R^3)
\ee
satisfies
\be\label{def-B-in1}
\mathcal{B}_\e(f)\in W_0^{1,2}(\Omega_\e),\quad \Div {B}_\e(f)=f \quad \mbox{in}\quad \Omega_\e.
\ee
Observe that the operator norm of $\mathcal{B}_\e$ relies on $\e$ only through the operator norm of $R_\e$ from $W_{0}^{1,2}(\Omega;\R^3)$ to $W_{0}^{1,2}(\Omega_\e;\R^3)$.  By Allaire's construction, the operator norm of $R_\e$ (partially) depends on the uniform estimates of a Stokes system in a ball with a shrinking hole.  Let $B_1=B(0,1)$ be the unit ball and $T\subset B_1$ be a model obstacle in the fluid which is assumed to be a simply connected $C^2$ domain. In accordance to the construction in \cite{ALL-NS1}, the operator norm of $R_\e$ depends on $\e$  partially through the $W_0^{1,2}$ estimate
 \begin{equation}\label{est-stokes0}
\|\nabla v\|_{L^2(B_1\setminus \e^{\alpha-1}T)}+\|q\|_{L^2_0(B_1\setminus \e^{\alpha-1}T)}\leq C_\e\, \|g\|_{L^2(B_1\setminus \e^{\alpha-1}T)}
\end{equation}
of the following Dirichlet problem of the Stokes equations
\begin{equation}\label{stokes}
-\Delta v+\nabla q=\Div g,\ \Div v =0 \quad \mbox{in}~=B_1\setminus \e^{\alpha-1}  T; \quad v=0 \quad \mbox{on} ~\d B_1\cup \e^{\alpha-1} \d T.
\end{equation}
Due to the $L^2$ framework, which represents the natural existence frame for the Stokes operator the constant $C_\e$ in \eqref{est-stokes0} is $1$ and hence independent of $\e$.   Furthermore, under the condition $\a\geq3$, it can be shown that the operator norm of $\mathcal{B}_\e$ defined in \eqref{def-B-in} is independent of $\e$. This implies uniform $L^2$ estimates of the pressure.
\smallskip


A natural thinking is to employ this construction of Bogovski{\u\i} operator and generalize it to the $L_0^r(\Omega_\e)$ and $ W_0^{1,r}(\Omega_\e;\R^3)$ framework for more general $r$.
It means a generalization of the restriction operator $R_\e$ to an operator from $W_0^{1,r}(\Omega;\R^3)\to W_0^{1,r}(\Omega_\e; \R^3)$. It was shown in \cite{L-hom1}) that this generalization can be done in a rather direct way by observing the well-posedness in $W_0^{1,r}(\Omega_\e; \R^3), 1<r<\infty$ of \eqref{stokes}.  The issue is that the operator norm of $R_\e:\ W_0^{1,r}(\Omega; \R^3)\to W_0^{1,r}(\Omega_\e; \R^3)$ depends on the $W_0^{1,r}$ estimate:
 \begin{equation}\label{est-stokes1}
\|\nabla v\|_{L^r(B_1\setminus \e^{\alpha-1}T)}+\|q\|_{L^r_0(B_1\setminus \e^{\alpha-1}T)}\leq C_\e\, \|g\|_{L^r(B_1\setminus \e^{\alpha-1}T)},
\end{equation}
for $v$ a solution to \eqref{stokes}.
However, for $r\neq 2$ the constant $C_\e$ depends on the Lipschitz character of the domain $B_1\setminus \e^{\alpha-1}T^s$ and is not uniform in $\e$ when $\a>1$. For the well-posedness result and estimates, we refer to \cite{BS} for details and the proof. It was the main motivation of the study in \cite{L-hom2} ({see also \cite{L-hom1} for the respective result for the Laplace operator}) to shown that the estimate constant $C_\e$ in \eqref{est-stokes1} is uniformly bounded in $\e$ for the range $3/2<r<3$. This allowed to construct a Bogovski{\u\i} operator that is uniformly bounded from $L_0^r(\Omega_\e)\to W_0^{1,r}(\Omega_\e; \R^3)$ for any $r\in (3/2,3)$ by using the construction \cite{FL1}.

Later a new line of apporach was developed that implied existence for solutions to the stationary compressible Navier--Stokes equations for the case $\g<3$   where the $L^2$ integrability of the pressure is no longer guaranteed (see \eqref{int-vr}). Observe that the restriction $r \in (3/2,3)$ corresponds to $\left(3-\frac{3}{\g}\right)'\in (3/2,3)$. This indicates that the existence frame for the stationary compressible setting can be extended to $\g>2$. Indeed, exactly for $\g>2$ it was shown in \cite{DFL} that a refinement to a more explicit construction of Bogovski{\u\i} type operator can be achieved. This builds the fundament of the Bogovski{\u\i} operator constructed in this article. For that reason the growth restriction on the pressure in this work are analogous to the restriction $\gamma>2$ in the stationary case. See Remark~\ref{rem-tech} for more details.

\subsubsection{New difficulties {and a technical theorem}}

In this paper, we turn to consider instationary compressible Navier--Stokes equations.  We will see later on that our main result obtained in this paper is formally the same as for the stationary Navier--Stokes equations; the techniques are rather different and there arise new difficulties.

The first difficulty comes from even lower integrability of the pressure. For instationary compressible Navier-Stokes equations with any $\gamma>3/2$,  the so far best integrability of the pressure one can obtain is (see \cite{LI4,F-book})
$
p(\vr)\in L^{\frac{5}{3}-\frac{1}{\gamma}},
$
where the range of the integrability component is $\frac{5}{3}-\frac{1}{\gamma}\in (1,\frac{5}{3})$. This is much worse than in the stationary setting and the deep reason why we have to impose the severe restriction $\g>6$ in Theorem \ref{Tm1}. It is to make sure that $\frac{5}{3}-\frac{1}{\gamma}>3/2$. More explanations on the restriction on $\g$ {and their connection to the "hole--size" $\alpha$} are given in Remark~\ref{rem-tech}.

Technically, the main new difficulty is the absence of uniform estimates for the restricted Bogovski{\u\i} operator $\mathcal{B}_\e$ in a negative Sobolev spaces. Indeed, due to the instationary setting, to obtain higher integrability of the pressure by using the Bogovski{\u\i} operator, one needs to be able to handle terms of the form
$
{\mathcal B}_\e(\Div(\vr^\th \vu))
$
which comes from the term ${\mathcal B}_\e(\d_t(\vr^\th))$ via the renormalized continuity equation. Here $\th$ is an exponent in the range  $(0,\frac{2\g}{3}-1]$. This enforces us to find uniform estimates of the following type
\be\label{est-B-in2}
\|{\mathcal B}_\e(\Div g)\|_{L^r(\Omega_\e; \R^3)} \leq C \|g\|_{L^r(\Omega_\e; \R^3)},\quad \mbox{for any $g\in L^r(\Omega_\e; \R^3),\ g\cdot {\bf n}=0$ on $\d\O_\e$},
\ee
with the constant $C$ independent of $\e$.

If we would like to employ the construction of $\mathcal{B}_\e$  in \cite{L-hom2} and obtain such an estimate \eqref{est-B-in2}, one would need to obtain uniform estimates for very weak solutions to \eqref{stokes}:
 \begin{equation}\label{est-stokes2}
\|v\|_{L^r(B_1\setminus \e^{\alpha-1}T)}+\|q\|_{W^{-1,r}(B_1\setminus \e^{\alpha-1}T)}\leq C \|g\|_{W^{-1,r}(B_1\setminus \e^{\alpha-1}T)},
\end{equation}
{with constant independent of $\e$. Above $T$ represents once more the shape of a model obstacle in the fluid.}
This is unknown according to the authors' knowledge and we think that the estimate is not valid for any $r>3/2$.

Our approach is different to the approach in \cite{FL1} and more general, following the idea of \cite{DFL}. It was possible due to a further and new refinement to the plethora of results related to the Bogovkii operator. Namely uniform estimates in negative spaces for so called John domains. John domains are a typical known generalization of Lipschitz domains on which bounds of the Bogovski{\u\i} operator in Lebesgue spaces could be shown. We refer to \cite{ADM, DieRuzSch10} for the definition of John domains and the construction of such Bogovski{\u\i} operators. We are able to prove that the very same Bogovski{\u\i} operator constructed in~\cite{DieRuzSch10} actually is bounded in negative Sobolev spaces as well. Since we believe this result is of independent interest we state the theorem here as a technical tool for further use:

\begin{Theorem}\label{thm-Bog}
Let $\Omega$ be a John-domain, or more general a domain that satisfies the emanating chain condition defined in~\cite[Def. 3.5]{DieRuzSch10}.  Let $C^\infty_{c,0}(\Omega)$ be the space of smooth and compactly supported functions with zero mean values in $\O$. Then for any $1<q<\infty$ there exists a linear operator $\mathcal{B} :  C^\infty_{c,0}(\Omega)\to C^\infty_{c}(\Omega;\R^3)$ such that:

If $f\in L^q(\Omega)$ with $\int_{\Omega} f\ \dx =0$, there holds
\be\label{thm-div1}
\Div (\mathcal{B}(f))=f\  \mbox{in} \ \Omega, \quad \|\mathcal{B}(f)\|_{W_{0}^{1,q}(\Omega;\R^3)}\leq C\|f\|_{L^q(\Omega)}
\ee
for some constant $C$ dependent only on the emanating chain constants.

{ If $f\in (W^{1,r'}(\Omega))\cap L^{r'}_0(\Omega))^*=\{g\in (W^{1,r'}(\Omega))^*\,:\, \langle g,1 \rangle=0\}$, then
\be\label{thm-div2}
\langle \mathcal{B}(f),\nabla \phi\rangle=\langle f, \phi \rangle \  \mbox{for} \ \phi\in W^{1,r'}(\Omega), \quad
 \|\mathcal{B}(f)\|_{L^{r}(\Omega;\R^3)}\leq C \|f\|_{(W^{1,r'}(\Omega))^*}.
\ee
}
\end{Theorem}

Observe, that \eqref{thm-div1} was shown in~\cite{DieRuzSch10}. The second bound \eqref{thm-div1} is shown in section 2 below.
\begin{Remark}
Let $\Omega$ be a John-domain, then we find for $f\in L^q(\Omega;\R^3)$ with
 $f\cdot {\bf n}=0$ on $\partial\Omega$ in the weak sense, that $\Div (\mathcal{B}(\Div f))= \Div f$. More explicit
\begin{align}
\label{thm-div3}
\begin{aligned}
\langle \mathcal{B}(\Div f), \nabla \phi\rangle =\skp{f}{\nabla \phi}\  \mbox{for} \ \phi\in W^{1,q'}(\Omega), \quad \| \mathcal{B}(\Div f)\|_{L^{q}(\Omega;\R^3)}\leq C \|f\|_{L^q(\Omega;\R^3)}.
\end{aligned}
\end{align}
This is due to the fact that $f\in L^q(\Omega;\R^3)$ and $f\cdot {\bf n}=0$ in the weak sense implies $\Div f\in \set{g\in (W^{1,r'}(\Omega))^*\,:\,\langle g,1 \rangle=0}$.

\end{Remark}
Below we will use the theorem to construct a restriction operator which is able to control the weak time derivative of the density.  As a consequence, a higher integrability of the density can be deduced.
%

\subsection{Weak solutions}

\begin{defi}[Finite energy weak solution]\label{def-weaksol}

We say that $(\vr, \vu)$ is a \emph{finite energy weak solution} of the Navier-Stokes equations \eqref{i1}-\eqref{i3} supplemented with the boundary condition \eqref{i4}, the assumption on pressure \eqref{pp1}, and the following initial conditions
\be\label{ic}
\vr(0, \cdot) = \vr_0 \in L^\g(\O_\e), \  \vu(0, \cdot) = \vu_0 \in L^{\frac{2\g}{\g-1}}(\O_\e;\R^3),
\ee
in the space-time cylinder $(0,T) \times \Omega_\e$ if:
\begin{itemize}

\item There holds:
\[\vr \geq 0 \ a.e.\  in \  (0,T) \times \Omega_\e,\quad \vr \in C_{\rm weak} (0,T;  L^\gamma(\Omega_\e)),\]
\[
 \vr \vu \in C_{\rm weak}(0,T; L^{\frac{2 \gamma}{ \gamma + 1}}(\Omega_\e;\R^3)),
\ \vu \in L^2(0,T; W_0^{1,2}(\Omega_\e;\R^3)).
\]

\item For any $0 \leq \tau \leq T$ and any test function $\varphi \in \DC([0,T] \times \Omega_\e)$:

\be\label{Mm11}
\int_0^\tau \intOe{ \left[ \vr \partial_t \varphi + \vr \vu \cdot \Grad \varphi \right] } \ \dt =
\intOe{ \vr(\tau, \cdot) \varphi (\tau, \cdot) } - \intOe{ \vr_{0} \varphi (0, \cdot) }.
\ee

\item For any $0 \leq \tau \leq T$ and any test function $\varphi \in \DC([0,T] \times {\Omega_\e};\R^3)$:
\ba\label{Mm2}
\int_0^\tau \intOe{ \left[ \vr \vu \cdot \partial_t \varphi + \vr \vu \otimes \vu : \Grad \varphi  + p(\vr) \Div \varphi \right] } \ \dt
\ea
\[
= \int_0^\tau \intOe{ {\mathbb S}(\Grad \vu) : \Grad \varphi -\vr \vc{f} \varphi } \ \dt + \intOe{ \vr \vu (\tau, \cdot) \cdot \varphi (\tau, \cdot) } - \intOe{ \vr_{0} \vu_{0} \cdot \varphi(0, \cdot) }.
\]

\item The \emph{energy inequality}
\be\label{Mm4}
\int_{\Omega_\e} \big[
\frac{1}{2} \vr |\vu|^2 +  P(\vr)  \big] (\tau, \cdot) \ \dx
+ \int_0^\tau \intOe{ {\mathbb S}(\Grad \vu) : \Grad \vu } \ \dt
\leq \int_{\Omega_\e} \big[
\frac{1}{2} \vr_{0} |\vu_{0} |^2 +  P(\vr_0) \big]  \dx
\ee
holds for a.a. $\tau \in (0,T)$, where we have set $P(\vr) := \vr \int_1^\vr \frac{ p(z) }{z^2} \ {\rm d}z.$

\item Moreover, a finite energy weak solution $(\vr, \vu)$ is said to be a \emph{renormalized weak solution} if
\be\label{ren}
\d_t b(\vr)+\Div\big(b(\vr)\vu\big)+\big(b'(\vr)\vr-b(\vr)\big)\Div \vu=0\ \mbox{in}\ \mathcal{D}'\big((0,T)\times \R^3\big),
\ee
 for any $b\in C^1([0,\infty))$. In \eqref{ren} $(\vr, \vu)$ are extended to be zero outside $\O_\e$.



\end{itemize}

\end{defi}

We give two remarks concerning the definition and the existence of finite energy weak solutions.
\begin{Remark}\label{rem-ini}
The integrability for the initial data in \eqref{ic} is imposed such that the initial energy on the right-hand side of \eqref{Mm4} is bounded.
\end{Remark}
\begin{Remark}\label{rem-exi}
For any fiexed $\e>0$, it can be shown by the theory developed by Lions \cite{LI4} and Feireisl-Novontn\'y-Petzeltov{\' a} \cite{F-book} that the Navier-Stokes equations \eqref{i1}-\eqref{i3}  admits global-in-time finite energy weak solutions $(\vr,\vu)$ for any finite energy initial data and pressure satisfying \eqref{pp1} with $\gamma>3/2$. Moreover, such solutions are also renormalized in the sense of \eqref{ren}.
\end{Remark}

\subsection{Main results}

Our main result is the following:

\begin{Theorem}\label{Tm1}

Let $(\vr_\e,\vu_\e)_{0<\e<1}$ be a family of finite energy weak solutions for the no-slip compressible Navier-Stokes equations \eqref{i1}-\eqref{i4} in $(0,T)\times \Omega_\e$ under the pressure condition \eqref{pp1} with $\gamma>6$ in the sense of Definition \ref{def-weaksol} with initial data satisfying
\be\label{ini-vr-vu}
 \vr_\e(0,\cdot)=\vr_{0,\e},\  \vu(0,\cdot)=\vu_{0,\e}, \quad  \sup_{0<\e<1} \left( \|\vr_{0,\e}\|_{L^\gamma(\Omega_\e)}+\|\vu_{0,\e}\|_{L^{3}(\Omega_\e;\R^3)}\right)= D < \infty.
\ee
This implies, up to a substraction of a subsequence, the weak convergence for the zero extensions as $\e \to 0$:
\be\label{con-vr-vu}
\tilde \vr_{0,\e} \to  \vr_0 \ \mbox{weakly in}\ {L^\gamma(\Omega)}, \quad \tilde \vu_{0,\e} \to  \vu_0 \ \mbox{strongly in}\ {L^3(\Omega;\R^3)}.\nn
\ee

Then there holds the uniform estimates for the solution family:
\be\label{uniform-est}
\sup_{0<\e<1}\big( \|\vr_{\e}\|_{L^\infty(0,T;L^\gamma(\Omega_\e))}+ \|\vr_{\e}\|_{L^{\frac{5\gamma}{3}-1}((0,T)\times \Omega_\e)}+\|\vu_{\e}\|_{L^2(0,T;W_0^{1,2}(\Omega_\e;\R^3))}\big)\leq  C(D)<\infty,
\ee
where $C(D)$ depends only on $D$. This implies, up to a substraction of a subsequence, that
\be\label{weak-limit}
\tilde \vr_\e  \to \vr \ \mbox{weakly-* in} \ L^\infty(0,T;L^\gamma(\Omega)),\quad \tilde \vu_\e \to \vu \   \mbox{weakly in}\  L^2(0,T; W_{0}^{1,2}(\Omega);\R^3).
\ee

If the growth parameter $\gamma$ in \eqref{pp1} and size parameter $\alpha$ of the holes in \eqref{dis-holes} satisfy
\be\label{ass-tec}
\frac{\gamma-6}{2\gamma-3} \cdot \alpha>3,
\ee
then the couple $(\vr,\vu)$ is the finite energy weak solution to \eqref{i1}-\eqref{i4} in $\Omega$  with initial data
\be\label{ini-limit}
\vr(0,\cdot)=\vr_0,\quad \vu(0,\cdot)=\vu_0.\nn
\ee
\end{Theorem}

We give some remarks on the technical condition \eqref{ass-tec}.

\begin{Remark}\label{rem-tech}

 This condition is needed due to the capacity estimate of the holes and the low integrability of the pressure $p(\vr_\e)$. By \eqref{pp1} and \eqref{uniform-est}, we have the uniform bound for the pressure:
\be\label{p-est}
\|p(\vr_\e)\|_{L^{\frac{5}{3}-\frac{1}{\gamma}}((0,T)\times\Omega_\e)}< C.
\ee

To extend the Navier-Stokes equations (in particular the momentum equation) from spatial domain $\Omega_\e$ to  $\Omega$,
  a idea is to find a family of functions $\{g_\e\}_{\e>0}$ vanishing on the holes and converges to $1$ in some Sobolev space $W^{1,q}(\Omega)$ for some $q$ determinant by the integrability of the pressure, and then to decompose any test function $\varphi \in C_c^\infty((0,T)\times \Omega; \R^3)$ as
\be\label{dec-test}
\varphi=g_\e \varphi+(1-g_\e) \varphi.
\ee
Then $g_\e \varphi$ can be treated as a test function for the momentum equation \eqref{i2} in $\Omega_\e$ provided all the terms in \eqref{Mm2} make sense. Here we need the condition on the largeness of $q$. For the terms related to the part $(1-g_\e) \varphi$, we show that they are small and converge to zero.

Firstly, such a function family $\{g_\e\}_{\e>0}$ exists provided the following condition satisfied (see Lemma \ref{g-exi}):
\be\label{ass-tec1}
(3-q)\alpha-3>0.
\ee
The quantity on the left-hand side of \eqref{ass-tec1} is determined by the uniform $q$-capacity assumptions of the union of all the holes:
\be\label{cap-holes}
{\rm Cap}_q\,\big(\bigcup_{k\in K_\e} T_{\e,k}\big)\leq \sum_{k\in K_\e} {\rm Cap}_q\,(T_{\e,k}) \leq C \e^{-3}\e^{\alpha(3-q)}.
\ee
In the weak formulation of the momentum equations, by testing $g_\e \varphi$, there arises the following term
$$
\int_0^\tau\intO{ p(\tilde \vr_\e) \varphi \cdot \nabla g_\e }\ \dt.
$$
To make sure that the above term makes sense, due to \eqref{p-est}, it is necessary to impose the condition
\be\label{ass-tec2}
\left(\frac{5}{3}-\frac{1}{\gamma}\right)^{-1}+\frac{1}{q}\leq 1.
\ee
The condition \eqref{ass-tec} is sufficient and also necessary to make sure that  there exists $5/2<q<3$ such that conditions \eqref{ass-tec1} and \eqref{ass-tec2} are satisfied. We also remark that condition \eqref{ass-tec2} implies
\be\label{ass-tec3}
\frac{1}{\gamma}+\frac{1}{3}+\frac{1}{q}\leq 1
\ee
for any $\gamma>3/2$ because
$
\left(\frac{5}{3}-\frac{1}{\gamma}\right)^{-1} \geq \frac{1}{\gamma}+\frac{1}{3} \Longleftrightarrow (2\gamma-3)^2\geq 0
$
which is always true. This is needed later on in the proof of Proposition \ref{prop-mom}.

\end{Remark}

\section{Model test functions and the Bogovski{\u\i} operator}

In this section, we introduce two basic tools that will be needed in our proof of the main theorem.

\subsection{Test functions vanishing on the holes}
We introduce the following Lemma:
\begin{Lemma}\label{g-exi}
For any $1<r<3$ such that $(3-r)\alpha-3>0$, there exits a family of functions $\{g_\e\}_{\e>0}\subset W^{1,r}(\Omega)$ such that
\be\label{g-fam}
g_\e =0 \quad \mbox{on}\ \bigcup_{k\in K_\e} T_{\e,k},\quad g_\e \to 1 \quad \mbox{in}\ W^{1,r}(\Omega).
\ee
Moreover, there holds the estimate for some $C$ independent of $\e$:
\be\label{g-est}
\| g_\e - 1\|_{ W^{1,r}(\Omega)}\leq C \e^{\sigma},\quad \sigma:=((3-r)\alpha-3)/r.
\ee
\end{Lemma}

\begin{proof}[Proof of Lemma \ref{g-exi}] By (\ref{dis-holes}), there exits $g_\e\in C^\infty (\R^3)$ such that
$$
g_\e=0 \ \mbox{on} \  \bigcup_{k\in K_\e} T_{\e,k},\quad g_\e=1 \ \mbox{on} \  \big(\bigcup_{k\in K_\e} B(x_{\e,k}, \de_0 \e^\a)\big)^c,\quad \|\nabla g_\e\|_{L^\infty(\R^3)}\leq C \e^{-\alpha}.
$$
Recall that the counting measure of $K_\e$ satisfies (\ref{number-holes}). Then direct calculation gives
$$
\| g_\e - 1\|_{ L^{r}(\Omega)}\leq C \e^{(3\alpha-3)/r},\quad \| \nabla g_\e \|_{ L^{r}(\Omega)}\leq C \e^{(3\alpha-3)/r-\alpha}=C \e^{\sigma}.
$$
Such $g_\e$ fulfill the request in Lemma \ref{g-exi}.

\end{proof}


\subsection{The Bogovski{\u\i} operator}\label{sec:Bog}

The aim of this section is to prove the following proposition:
\begin{Proposition}\label{prop-Bog}
Let $\Omega_\e$ defined as in \eqref{dis-holes}-\eqref{domain} with $\alpha\geq 1$, then for any $1<q<\infty$, there exists a linear operator $\mathcal{B}_\e \,: \, L^q(\Omega_\e)\to W_{0}^{1,q}(\Omega_\e;\R^3)$ such that for any $f\in L^q(\Omega_\e)$ with $\int_{\Omega_\e} f\ \dx =0$, there holds
\be\label{pro-div1}
\Div \mathcal{B}_\e(f)=f\  \mbox{in} \ \Omega_\e, \quad \|\mathcal{B}_\e(f)\|_{W_{0}^{1,q}(\Omega_\e;\R^3)}\leq C(1+\e^{((3-q)\a-3)/q}) \|f\|_{L^q(\Omega_\e)}
\ee
for some constant $C$ independent of $\e$.

{  For any $r >3/2$, the linear operator $\mathcal B_\e$ can be extended as a linear operator from $\{\Div \ggg : \vc{g}\in L^r(\Omega_\e;\R^3),
  \vc{g} \cdot {\bf n}=0 \mbox{ on } \d \O_\e\}$ to $L^r(\O_\e;\R^3)$ satisfying
\be\label{pro-div2}
 \|\mathcal{B}_\e(\Div {\bf g})\|_{L^{r}(\Omega_\e;\R^3)}\leq C \|{\bf g}\|_{L^r(\Omega_\e;\R^3)},
\ee
for all some constant $C$ independent of $\e$.}
%
\end{Proposition}
The existence of an operator such that the first equation in (\ref{pro-div1}) is satisfied is classical, known as Bogovski{\u\i} operator. The key point of this lemma is to give an estimate for the operator norm, in particular, to show the dependency on $\e$.

The analytical tool to be able to get estimates uniformly in $\e$ is Theorem~\ref{thm-Bog}.  So we prove Theorem \ref{thm-Bog} first and then prove Proposition \ref{prop-Bog}.

\begin{proof}[Proof of Theorem \ref{thm-Bog}]

The first statement \eqref{thm-div1} is just~\cite[Thm. 5.2]{DieRuzSch10}. Hence we will show, that the same operator also satisfies \eqref{thm-div2}.

We start by using the fact that compactly supported functions are dense in
$$\set{ f\in (W^{1,r'}(\Omega))^*\,:\,\skp{f}{1} = 0}.$$
Indeed, as can be seen in \cite[Lemma 10.4]{F-N-book} there exists
$w\in L^r(\Omega)$, such that
\[
\skp{f}{\phi}:=\int_{\Omega}w\cdot\nabla\phi\,dx,\text{ and }\norm{w}_{L^r(\Omega)}=\norm{f}_{(W^{1,r'}(\Omega))^*}.
\]
We approximate (using the density of test functions in Lebesgue spaces) with functions
$w_\delta\in C^\infty_c(\Omega;\R^3)$, such that $w_\delta\to w$ in $L^r(\Omega)$ and $\norm{w_\delta}_{L^r(\Omega)}\leq 2\norm{f}_{(W^{1,r'}(\Omega))^*}$.
%
 By partial integration we find $\Div w_\delta\in C^\infty_{c,0}(\Omega)$.

{  This allows us to apply the decomposition theorem~\cite[Thm. 4.2]{DieRuzSch10}.} I.e. for the chain covering $\mathcal{W}$ introduced in~\cite[Def. 3. 11]{DieRuzSch10}  we have the family of linear operators $T_i:C^\infty_{c,0}(\Omega)\to C^\infty_{c,0}(W_i)$, for all $W_i\in\mathcal{W}$. By the definition of the reference paper we find that $\mathcal{B}(\Div w_\delta):=\sum_{W_i\in\mathcal{W}}B_i T_i\Div w_\delta$ is well defined. Observe, that $B_i$ is the standard Bogovskij operator on the cube $W_i$.
Since, $ T_i \Div w_\delta\in C^\infty_{c,0}(W_i)$ we find by \cite[Theorem 10.11]{F-N-book} and~\cite[Thm. 4.2]{DieRuzSch10} as well as the finite intersection property of the $W_i$, that
\begin{align*}
\|\mathcal{B}\Div w_\delta\|_{L^r(\Omega)}^r & \leq C \sum_{W_i\in\mathcal{W}} \|B_i\Div T_i w_\delta\|_{L^r(\Omega)}^r\leq  C \sum_{W_i\in\mathcal{W}}\|T_i  w_\delta\|_{L^r(\Omega;\R^3)}^r
\\
&\leq  C \|w_\delta\|_{L^r(\Omega;\R^3))}^r\leq  C \|f\|_{(W^{1,r'}(\Omega))^*)}^r.
\end{align*}
Moreover, also by \cite[Theorem~4.2]{DieRuzSch10}, as well as the property of $B_i$, we find
\begin{align}
\Div \mathcal{B} w_\delta=\sum_{W_i\in\mathcal{W}} \Div B_i T_i w_\delta = \sum_{W_i\in\mathcal{W}} T_i w_\delta =w_\delta.
\end{align}
The operator $\mathcal{B}$ can be extended accordingly to $\{f\in (W^{1,r'}(\Omega))^*\,:\,\skp{f}{1}=0\}$ by letting $\delta\to 0$.
\end{proof}

\begin{proof}[Proof of Proposition \ref{prop-Bog}] Let $f\in L^q(\Omega_\e)$ with $\int_{\Omega_\e} f\ \dx=0$,  we consider the extension function $\tilde f=E(f)$ defined as
\be\label{tilde-f}
\tilde f=f \ \mbox{in}\ \Omega_\e,\quad \tilde f=0 \ \mbox{on}\  \Omega\setminus \Omega_\e=  \bigcup_{k\in K_\e} T_{\e,k}.
\ee
Then by employing the classical Bogovsii's operator, there exists $u={\mathcal B}(\tilde f)\in W_0^{1,q}(\Omega;\R^3)$ such that
\be\label{div-f-u}
\Div u=\tilde f \ \mbox{in} \ \Omega \quad \mbox{and} \quad \|u\|_{W_0^{1,q}(\Omega;\R^3)}\leq C\|\tilde f\|_{L^q(\Omega)}= C\|f\|_{L^q(\Omega_\e)}
\ee
for some constant $c$ depends only on $\Omega$ and $q$.

In each ball $B(x_{\e,k}, \de_3\e)$, we introduce two cut-off functions $\chi_k$ and $\phi_k$ such that:
\ba\label{cut-off-k}
&\chi_k\in C_c^\infty( B(x_{\e,k}, \de_3\e)), \ \chi_k=1 \ \mbox{on}\  \overline{B(x_{\e,k}, \de_2\e)}, \  |\nabla_x\chi_k|\leq C \e^{-1}; \\
 &\phi_k\in C_c^\infty(B(x_{\e,k}, \de_1\e^\alpha)), \ \phi_k=1 \ \mbox{on}\ \overline{B(x_{\e,k}, \de_0\e^\alpha)}, \  |\nabla_x\phi_k|\leq C \e^{-\alpha}.
\ea
We denote
\ba\label{cut-off-k1}
D_{\e,k}:=B(x_{\e,k}, \de_3\e)\setminus B(x_{\e,k}, \de_2\e), \quad E_{\e,k}:= B(x_{\e,k}, \de_3\e)\setminus T_{\e,k}.
\ea
We then define two localizations of $u$ near the holes:
\ba\label{cut-off-k2}
b_k(u) := \chi_k \left(u-\langle u\rangle_{D_{\e,k}}\right)\in W_0^{1,q}(B(x_{\e,k}, \de_3\e)),\ \b_k(u) := \phi_k \langle u\rangle_{D_{\e,k}} \in W_0^{1,q}(B(x_{\e,k}, \de_1\e^\alpha)),
\ea
where
$$\langle u\rangle_{D_{\e,k}}:=\frac{1}{|D_{\e,k}|}\int_{D_{\e,k}} u\  \dx.$$

Since the $E_{\e,k}$ are uniform John domains, the following corollary follows immediately by Theorem~\ref{thm-Bog} and \eqref{thm-div3}.
\begin{Corollary}\label{lem-Bog1}
For any $1<q<\infty$, there exit a linear operator $\mathcal{B}_{E_{\e,k}}: L_0^q(E_{\e,k}) \to W_0^{1,q}(E_{\e,k};\R^3)$
such that for any $ f\in L_0^q(E_{\e,k})$
\ba\label{bog-E-F1}
&\Div \mathcal{B}_{E_{\e,k}}(f)= f,\quad \| \mathcal{B}_{E_{\e,k}}( f) \|_{W_0^{1,q}(E_{\e,k};\R^3)}\leq C \|f\|_{L^q(E_{\e,k})},
\ea
for some constant $C$ independent of $\e$.

 For any $\ggg\in L^q(E_{\e,k};\R^3)$ with $\ggg \cdot {\bf n}=0$ on $\d E_{\e,k}$, we find
 \[
\skp{\mathcal{B}_{E_{\e,k}}(\Div \ggg)}{\nabla \phi}=\skp{\ggg}{\nabla \phi}\text{ for all }\phi\in C^\infty(E_{\e,k}).
\]
Moreover,
\ba\label{bog-E-F2}
&\|\mathcal{B}_{E_{\e,k}}(\Div \ggg) \|_{L^q(E_{\e,k};\R^3)}\leq C \|\ggg\|_{L^q(E_{\e,k};\R^3)},
\ea
for some constant $C$ independent of $\e$.

\end{Corollary}


\medskip

We now define the restriction operator in the following way:
\be\label{restr-1}
R_\e (u) := u- \sum_{k\in K_\e}  \left( b_k(u) +\beta_k (u) - {\mathcal B}_{E_{\e,k}}(\Div ( b_k(u) + \beta_k(u) ) \right).
\ee
We first check that $R_\e (u)$ is well defined. Since $b_k,\beta_k\in W^{1,q}(E_{\ep,k})$ to this end, it is sufficient to show that
\be\label{zero-mean1}
\int_{E_{\e,k}} \Div (b_k+\beta_k)\,\dx =0.
\ee
Indeed, on one hand, by the zero trace \eqref{cut-off-k2}, we have
\be\label{zero-mean2}
\int_{ B(x_{\e,k}, \de_3\e)} \Div (b_k+\beta_k)\,\dx =0.
\ee
On the other hand, by \eqref{tilde-f}-\eqref{cut-off-k2}, we have
\ba\label{zero-mean3}
&\Div (b_k+\beta_k) = \chi_k \Div u+ \nabla_x \chi_k \cdot (u-\langle u\rangle_{D_{\e,k}})+\nabla_x \phi_k\cdot \langle u\rangle_{D_{\e,k}} =0, \quad  \mbox{on} \quad T_{\e,k}.
\ea
Equations \eqref{zero-mean2} and \eqref{zero-mean3} imply \eqref{zero-mean1}. Hence, we can apply Corollary~\ref{lem-Bog1} on each member and \eqref{restr-1} is well defined.
Moreover, applying Poincar\'e inequality gives
\begin{align*}
\int_{\Omega} \abs{\nabla R_\e(u)}^q\, \dx
&\leq C \int_{\Omega} \abs{\nabla u}^q\, \dx
+C \sum_{k\in K_\e}\int_{B(x_{\e,k}, \de_3\e)}\chi_k\abs{\nabla u}^q+\abs{u-\langle{u}\rangle_{D_{\e,k}}}^q\abs{\nabla \chi}^q\, \dx
\\
&\quad+ C \sum_{k\in K_\e}\int_{B(x_{\e,k}, \de_1\e^\alpha)}\abs{\nabla_x \phi_k}^q\abs{\langle{u}\rangle_{D_{\e,k}}}^q\, \dx
\\
&\leq  C \int_{\Omega} \abs{\nabla u}^q\, \dx
+C \sum_{k\in K_\e}\int_{D_{\ep,k}}\Bigabs{\frac{u-\langle{u}\rangle_{D_{\e,k}}}{\ep}}^q\, \dx + C \sum_{k\in K_\e}\abs{\langle{u}\rangle_{D_{\e,k}}}^q |B(x_{\e,k}, \de_1\e^\alpha)|
\\
&\leq  C \int_{\Omega} \abs{\nabla u}^q\, \dx + C \sum_{k\in K_\e}\int_{D_{\ep,k}}\abs{\nabla u}^q\, \dx + C \sum_{k\in K_\e}\abs{\langle{u}\rangle_{D_{\e,k}}}^q\ep^{(3-q)\alpha}\, \dx
\\
&\leq  C \int_{\Omega} \abs{\nabla u}^q\, \dx
+ C \sum_{k\in K_\e}\int_{D_{\e,k}}\abs{u}^q\, \dx \, \e^{(3-q)\alpha-3}
\\
&\leq C(1+ \e^{(3-q)\alpha-3})\norm{u}_{W^{1,q}(\Omega)}^q.
\end{align*}

We claim that the operator ${\mathcal B}_\e$ defined in the following way fulfills the desired properties in Proposition \ref{prop-Bog}:
\be\label{def-B-e}
{\mathcal B}_\e(f):= R_\e (u)= R_\e ({\mathcal B} (\tilde f))= R_\e \circ {\mathcal B} \circ E (f).\nn
\ee

By the definition of $R_\e (u)$ in \eqref{restr-1} and the property of $u$ in \eqref{div-f-u}, we have
\be\label{pt-Ru1}
R_\e (u) \in W_0^{1,q}(\O;\R^3),\quad \Div R_\e (u)= \Div u =\tilde f \quad \mbox{in $\O$}.\nn
\ee
Moreover, for any $x\in T_{\e,k}$ for some $k\in K_\e$, we have by using  \eqref{cut-off-k} and \eqref{cut-off-k1} that
\ba\label{pt-Ru2}
R_\e (u)(x) = u(x)- \chi_k(x) \left(u(x)-\langle u\rangle_{D_{\e,k}}\right)- \phi_k (x)\langle u\rangle_{D_{\e,k}}=0.\nn
\ea
Thus,
\be\label{pt-Ru3}
R_\e (u) \in W_0^{1,q}(\O_\e;\R^3),\quad \Div R_\e (u)= f \quad \mbox{in $\O_\e$}.\nn
\ee

\medskip

Now we prove the second part of Proposition \ref{prop-Bog}.  Let $r>3/2$ and $\vc{g}\in L^r(\Omega_\e;\R^3)$ with $\vc{g} \cdot {\bf n}=0$ on $\d \O_\e$.  Let $u:=\mathcal B  (\Div \tilde \ggg)\in  L^r(\O;\R^3) $ with $\mathcal B$ be the standard Bogovski{\u\i} operator in $\O$. Moreover, there holds
\be\label{est-B-gg}
\|u\|_{L^r(\O;\R^3)} \leq C \|\ggg\|_{L^r(\O_\e;\R^3)},
\ee
where the constant $C$ only depends on $r$ and the Lipschitz character of the domain $\O$.

We assume $\Div \ggg \in L^q(\O_\e)$ for some $q>1$. This ensures, together with the assumption $\bfg\cdot\bfn=0$ on $\partial \Omega$, that $ u\in W_0^{1,q}(\O;\R^3)$, then $R_\e(u)\in W_0^{1,q}(\O_\e;\R^3)$ is well defined where $R_\e$ is the restriction operator constructed by \eqref{restr-1}. Our goal is to show the following uniform estimates
\be\label{R-u-est}
\| R_\e(u) \|_{L^{r}(\O_\e;\R^3)} \leq C \| \ggg \|_{L^{r}(\O_\e;\R^3)},
\ee
for some constant $C$ independent of $\e$ and $\|\Div\ggg\|_{L^q(\O_\e)}$.

By \eqref{cut-off-k2} and \eqref{cut-off-k}, we have
\ba\label{est-bk}
& \int_\O | b_k(u)|^r \,\dx \leq \int_{B(x_{\e,k},\de_3 \e)} |u-\langle u\rangle_{D_{\e,k}}|^r\,\dx \leq C \int_{B(x_{\e,k},\de_3 \e)} |u|^r\,\dx,\\
&\int_\O |\b_k(u)|^r \,\dx  \leq  \int_{B(x_{\e,k},\de_1 \e^\a)} | \langle u\rangle_{D_{\e,k}}|^r \,\dx \leq C \e^{3(\a -1)}\int_{B(x_{\e,k},\de_3 \e)} | u |^r \,\dx,
\ea
where we used the fact
$
| \langle u\rangle_{D_{\e,k}}|^r \leq \langle |u|^r \rangle_{D_{\e,k}}.
$

Direct calculation gives, using the fact that $\Div \ggg=\Div u$, that
\ba\label{div-bk-u1}
\Div (b_k (u)+\beta_k (u)) &= \chi_k \Div u + \nabla_x \chi_k \cdot (u-\langle u \rangle_{D_{\e,k}}) + \nabla_x \phi_k \cdot \langle u \rangle_{D_{\e,k}} \\
&= \chi_k \Div  \ggg + \nabla_x \chi_k \cdot (u-\langle u \rangle_{D_{\e,k}}) + \nabla_x \phi_k \cdot \langle u \rangle_{D_{\e,k}}\\
&=  \Div (\chi_k \ggg) - \nabla_x\chi_k \cdot \ggg + \nabla_x \chi_k \cdot (u-\langle u \rangle_{D_{\e,k}}) + \nabla_x \phi_k \cdot \langle u \rangle_{D_{\e,k}}
\\
&=  \Div (\chi_k \ggg) + \Div(\chi_k (u-\ggg)) + \Div((\phi_k-\chi_k)\langle u \rangle_{D_{\e,k}}).
\ea
We will estimate the operator $\mathcal{B}_{E_{\e,k}}$ on each three divergences separately.
Observing
$$\chi_k\tilde \ggg = 0 \ \mbox{on} \ \d B(x_{\e,k},\de_3 \e), \quad  \chi_k\tilde \ggg \cdot {\bf n} =0 \ \mbox{on}\ \d T_{\e,k},$$
and applying Corollary \ref{lem-Bog1} implies
\be\label{div-bk-u2}
\|\mathcal{B}_{E_{\e,k}}(\Div (\chi_k \ggg))\|_{L^r(E_{\e,k};\R^3)} \leq C \|\chi_k \ggg\|_{L^r(E_{\e,k};\R^3)}\leq C\| \ggg \|_{L^r(E_{\e,k};\R^3)}.
\ee

For the other two terms we will use the $W^{1,q}$-bounds of the operator $\mathcal{B}_{E_{\e,k}}$.
We find by the support properties of $\chi, \ \phi$  and the fact $\Div u = \Div \ggg$  that
\[
\int_{E_{\e,k}}\Div(\chi_k (u-\ggg))\, dx = \int_{B(x_{\e,k},\de_3 \e)}\Div(\chi_k (u-\ggg))\, dx  = 0,
\]
\[
\int_{E_{\e,k}}\Div((\phi_k-\chi_k)\langle u \rangle_{D_{\e,k}})\,dx = \int_{B(x_{\e,k},\de_3 \e)} \Div((\phi_k-\chi_k)\langle u \rangle_{D_{\e,k}})\,dx = 0.
\]
Since $r>3/2$, we let $\tilde r>1 $ such that $1/\tilde r = 1/r + 1/3 $. Then applying Corollary \ref{lem-Bog1}, together with the Sobolev embedding inequality, implies
\ba\label{div-bk-u3}
&\|\mathcal{B}_{E_{\e,k}}(\nabla_x\chi_k \cdot (\ggg-u))\|_{L^r(E_{\e,k};\R^3)} \leq C \|\mathcal{B}_{E_{\e,k}}(\nabla_x\chi_k \cdot (\ggg-u))\|_{W^{1,\tilde r}_0(E_{\e,k};\R^3)} \\
&\quad \leq  C \|\nabla_x\chi_k \cdot (\ggg-u)\|_{L^{\tilde r}(E_{\e,k})} \leq C \|\nabla_x\chi_k\|_{L^{3}(E_{\e,k};\R^3)} (\|\ggg\|_{L^{r}(E_{\e,k};\R^3)}+ \|u\|_{L^{r}(E_{\e,k};\R^3)}),\\
&\|\mathcal{B}_{E_{\e,k}}(\nabla_x (\phi_k-\chi_k) \cdot \langle u \rangle_{D_{\e,k}})\|_{L^r(E_{\e,k};\R^3)}\\
&\quad \leq C (\|\nabla_x\chi_k\|_{L^{3}(E_{\e,k};\R^3)}+ \|\nabla_x\phi_k\|_{L^{3}(E_{\e,k};\R^3)})\|u\|_{L^{r}(E_{\e,k};\R^3)}\leq C \|u\|_{L^{r}(D_{\e,k};\R^3)}.
\ea

Summarizing \eqref{est-bk}--\eqref{div-bk-u3}, using \eqref{est-B-gg} and the fact $\|\nabla_x\chi_k\|_{L^{3}(E_{\e,k};\R^3)}+ \|\nabla_x\phi_k\|_{L^{3}(E_{\e,k};\R^3)}\leq C$, we obtain
\be\label{R-u-est-f}
\| R_\e(u) \|_{L^{r}(\O_\e;\R^3)}^r \leq C (\| \ggg \|^r_{L^{r}(\O_\e;\R^3)}  + \| u \|^r_{L^{r}(\O;\R^3)} )\leq C \| \ggg \|^r_{L^{r}(\O_\e;\R^3)}.\nn
\ee
for some constant $C$ independent of $\e$ and $\|\Div\ggg\|_{L^q(\O_\e)}$. This implies our desired estimate \eqref{R-u-est}. Since this estimate constant $C$ is independent of $\|\Div\ggg\|_{L^q(\O)}$, by a density argument we can eliminate the assumption $\Div\ggg\in L^q(\O)$. The proof is completed.

\end{proof}


\section{Proof of Theorem \ref{Tm1}}

\subsection{Uniform bounds}

By the uniform bounds of the initial data in \eqref{ini-vr-vu} and the energy inequality (\ref{Mm4}), together with the pressure property (\ref{pp1}),  Korn's inequality and Poincar\'e inequality, we have
\be\label{bd-vr}
\{\tilde \vr_\e\}_{\e>0} \ \mbox{uniformly bounded in}\ L^\infty(0,T;L^\gamma(\Omega)),
\ee
\be\label{bd-vru}
\{\tilde \vr_\e|\tilde\vu_\e|^2\}_{\e>0} \ \mbox{uniformly bounded in}\ L^\infty(0,T;L^1(\Omega)),
\ee
\be\label{bd-vu}
\{\tilde\vu_\e\}_{\e>0} \ \mbox{uniformly bounded in}\ L^2(0,T;W_0^{1,2}(\Omega;\R^3)).
\ee
Then up to a substraction of subsequence, we derive the weak limit
\be\label{weak-limit0}
\tilde \vr_\e  \to \vr \ \mbox{weakly-* in} \ L^\infty(0,T;L^\gamma(\Omega)),\quad \tilde \vu_\e \to \vu \   \mbox{weakly in}\  L^2(0,T; W_{0}^{1,2}(\Omega);\R^3).\nn
\ee

\subsection{Improved integrability on the density}

In this subsection, we will show the following improved uniform integrability of the density:
\be\label{bd-vr2}
\{\tilde\vr_\e\}_{\e>0} \ \mbox{uniformly bounded in}\ L^{\frac{5\gamma}{3}-1}((0,T)\times\Omega).
\ee
This is equivalent to following integrability on the pressure:
\be\label{bd-vr-p}
 \{p(\tilde\vr_\e)\}_{\e>0} \ \mbox{uniformly bounded in}\ L^{\frac{5}{3}-\frac{1}{\gamma}}((0,T)\times \Omega).
\ee

We recall that, for any fixed $\e>0$, the result $\vr_\e \in L^{\frac{5\gamma}{3}-1}((0,T)\times\Omega)$ is known, see for instance Theorem 7.7 in \cite{N-book}. However, the estimate for $\| \vr_\e \|_{L^{\frac{5\gamma}{3}-1}((0,T)\times\Omega)}$ shown in \cite{N-book} is not uniformly bounded in $\e$. Indeed, it depends on the Lipschitz norm of the spatial domain, which is the perforated domain $\O_\e$ in our setting.  Our task in this section is to show such an estimate is uniform in $\e$.

The idea (as in \cite{FNP, N-book}) is to test the the momentum equation
by  functions of the form
\be\label{test-Be}
\varphi(t,x):=\psi(t)\mathcal{B}_\e\left(\vr_\e^{\theta}-\langle \vr_\e^\theta\rangle\right),\ \ \psi \in C_c^\infty((0,T)),\ \langle \vr_\e^\theta\rangle :=\frac{1}{|\Omega_\e|}\intOe {\vr_\e^{\theta}},\ \theta>0,
\ee
where $\mathcal{B}_\e$ is the Bogovski{\u\i} operator given in Proposition \ref{prop-Bog}. To prove (\ref{bd-vr2}), a nature choice is $\theta=2\gamma/3-1$ in (\ref{test-Be}), but due to the restriction ($3/2<q\leq 2$) on the uniform bound of $\mathcal{B}_\e$, there arise terms that cannot be controlled by the known uniform estimates in (\ref{bd-vr})--(\ref{bd-vu}). We will improve the integrability of $\vr_\e$ step by step by taking $\theta$ from smaller number $\gamma/2$ to our desired number $2\gamma/3-1$.

The following two lemmas are needed:
\begin{Lemma}\label{lem-ren1}
Under the assumption in Theorem 1.1, the extension functions  $\tilde\vr_\e,\tilde \vu_\e$  satisfy
 \be\label{eq-tvr1}
 \d_t \tilde \vr_\e+\Div(\tilde \vr_\e \tilde \vu_\e)=0,\quad \mbox{in}\ \mathcal{D}'((0,T)\times \R^3).
 \ee
\end{Lemma}

\begin{Lemma}\label{lem-ren2}
Let $2\leq \beta<\infty$  and $\vr\in L^\beta(0,T;L^\beta_{loc}(\R^3))$, $\vr \geq 0$ a.e. in $(0,T)\times \R^3$, $\vu \in L^2(0,T;W^{1,2}_{loc}(\R^3;\R^3))$. Suppose that
 \be\label{eq-tvr2}
 \d_t \vr+\Div(\vr \vu)=0,\quad \mbox{in}\ \mathcal{D}'((0,T)\times \R^3).\nn
 \ee
Then for any $b\in C^0([0,\infty))\cap C^1((0,\infty))$ such that
\be\label{b-prop1}
 b'(s)\leq C s^{-\lambda_0} \ \mbox{for}\  s\in (0,1],\quad \ b'(s)\leq Cs^{\lambda_1} \ \mbox{for}\  s\in [1,\infty),\nn
 \ee
with
$
c>0,\ \lambda_0<1, \ \lambda_1 \leq \frac{\beta}{2}-1,
$
there holds the following renormalized equation
\be\label{reno-eq}
\d_t b(\vr)+\Div\big(b(\vr)\vu\big)+\big(b'(\vr)\vr-b(\vr)\big)\Div \vu=0\ \mbox{in}\ \mathcal{D}'\big((0,T)\times \R^3\big).\nn
\ee
\end{Lemma}

The proof of Lemma \ref{lem-ren1} is the same as Proposition \ref{prop-con} and the proof is postponed. For the proof of Lemma \ref{lem-ren2}, we refer to Lemma 6.9 in \cite{N-book}

By the estimate (\ref{bd-vr}) with $\gamma>6$ and the estimate (\ref{bd-vu}), we apply the above two lemmas to obtain
\be\label{eq-vr-b}
\d_t b(\tilde \vr_\e)+\Div\big(b(\tilde \vr_\e)\tilde \vu_\e\big)+\big(b'(\tilde \vr_\e)\tilde \vr_\e - b(\tilde \vr_\e)\big)\Div \tilde \vu_\e=0\ \mbox{in}\ \mathcal{D}'\big((0,T)\times \R^3\big),
\ee
for any $b$ as in Lemma \ref{lem-ren2}; in particular, we can take $b(s)=s^\th$ with $0< \th \leq \frac{\gamma}{2}$.

We first use the function (\ref{test-Be}) with $\th:={\gamma}/{2}$ as a test function in (\ref{Mm2}). Thus, by uniform estimate \eqref{bd-vr} and Proposition \ref{prop-Bog}, we have for any $r\in (1,2]$ that
\be\label{B-vr-theta-1}
 \|\mathcal{B}_\e(\vr^\th- \langle \vr^\th \rangle)\|_{L^\infty(0,T;W_{0}^{1,r}(\Omega_\e;\R^3))}\leq C \| (\vr^\th- \langle \vr^\th \rangle)\|_{L^\infty(0,T;L^r(\Omega_\e))} \leq  C \| \vr\|_{L^\infty(0,T;L^\gamma(\Omega_\e))}^{\frac{\g}{2}}\leq C.
\ee

Then direct calculation gives
\begin{equation}\label{p-Ij}
\int_0^T \intOe{ \psi(t) p(\vr_\e)\vr_\e^\theta}\,\dt = \sum_{j=1}^6 I_j,\nn
\end{equation}
where
\ba\label{Ij}
&I_1:=\int_0^T \intOe{ \psi p(\vr_\e)\langle\vr_\e^\theta\rangle}\,\dt,\ I_2:=-\int_0^T \intOe{ \d_t\psi(t)\vr_\e \vu_\e \cdot \mathcal{B}_\e\big(\vr_\e^\theta-\langle\vr_\e^\theta\rangle \big)}\,\dt,\\
&I_3:=-\int_0^T \intOe{ \psi\vr_\e \vu_\e \otimes \vu_\e : \nabla_x \mathcal{B}_\e\big(\vr_\e^\theta-\langle\vr_\e^\theta\rangle \big)}\,\dt, \ I_4:=\int_0^T \intOe{ \psi{\mathbb S} (\Grad \vu_\e) : \nabla_x \mathcal{B}_\e\big(\vr_\e^\theta-\langle\vr_\e^\theta\rangle \big)}\,\dt, \\
&I_5:=-\int_0^T \intOe{ \psi\vr_\e\vu_\e \cdot \mathcal{B}_\e\big(\d_t\vr_\e^\theta-\d_t\langle\vr_\e^\theta\rangle \big)}\,\dt,\ I_6:= \int_0^T \intOe{\psi \vr_\e {\bf f} \cdot \mathcal{B}_\e\big(\vr_\e^\theta-\langle\vr_\e^\theta\rangle\big)}\,\dt.\nn
\ea


We estimate using the choice $\theta=\gamma/2$ and H\"older's inequality
\begin{align*}
\abs{I_1}\leq C \sup_{t\in [0,T]}\abs{\langle\vr_\e^\theta(t)\rangle}\intOe{ \abs{ p(\vr_\e(t)}}\leq C\norm{\vr_\e}_{L^\infty(0,T;L^\gamma(\Omega_\e))}^{3\gamma/2}\leq C.
\end{align*}

By the fact $\frac1\gamma+\frac16+\frac16\leq 1$, we can estimate with estimate \eqref{B-vr-theta-1} and H\"older's inequality
\begin{align*}
\abs{I_2}\leq C\norm{\vr_\e}_{L^\infty(0,T;L^\gamma(\Omega_\e))}\norm{\bfu_\e}_{L^2(0,T;L^6(\Omega_\e))}\norm{ \mathcal{B}_\e\left(\vr_\e^\theta-\langle\vr_\e^\theta\rangle \right)}_{L^\infty(0,T;L^6(\Omega_\e))}\leq C.
\end{align*}

Again by the fact $\gamma>6$ and $\frac1\gamma+\frac13+\frac12\leq 1$, we obtain with estimate \eqref{B-vr-theta-1} and H\"older's inequality
\begin{align*}
\abs{I_3}\leq C\norm{\vr_\e}_{L^\infty(0,T;L^\gamma(\Omega_\e))}\norm{\bfu_\e}_{L^2(0,T;L^6(\Omega_\e))}^2\norm{\Grad \mathcal{B}_\e\left(\vr_\e^\theta-\langle\vr_\e^\theta\rangle \right)}_{L^\infty(0,T;L^2(\Omega_\e))}\leq C.
\end{align*}

Similarly, for $I_4$, we have
\ba
|I_4| &\leq C\norm{\nabla_x\bfu_\e}_{L^2(0,T;L^2(\Omega_\e))} \norm{\Grad \mathcal{B}_\e\left(\vr_\e^\theta-\langle\vr_\e^\theta\rangle \right)}_{L^\infty(0,T;L^2(\Omega_\e))} \\
& \leq C \norm{\vr_\e^\theta}_{L^\infty(0,T;L^2(\Omega_\e))}\leq C\norm{\vr_\e}_{L^\infty(0,T;L^\g(\Omega_\e))}^\theta\leq C.\nn
\ea

The most challenging term is the term which involves the weak time derivative. To estimate this term we had to introduce the theory for Bogovski{\u\i} type operator on negative spaces for perforated domains, namely the second part of Proposition \ref{prop-Bog}, for which the proof needs the theory for Bogovski{\u\i} type operator on negative spaces for John domains, namely Theorem \ref{thm-Bog}.
First we may use Lemma \ref{lem-ren2} with $\beta=\gamma$ to obtain
\be\label{eq-rth3}
\d_t (\vr^{\theta})+\Div\big(\vr^\theta \vu\big)+\big((\theta-1)\vr^\theta\big)\Div \vu=0\ \mbox{in}\ \mathcal{D}'\big((0,T)\times \R^3\big), \quad \theta:=\frac{\gamma}{2}.\nn
\ee
Then, direct calculation gives
\ba\label{est-B-dt-vr-1}
I_5 &= \int_0^T \intOe{ \psi\vr_\e\vu_\e \cdot  \mathcal{B}_\e\left(\Div\big(\vr_\e^\theta \vu_\e \big) \right) }\,\dt \\
&\qquad + (\th-1)\int_0^T \intOe{ \psi\vr_\e\vu_\e \cdot  \mathcal{B}_\e\left(\vr_\e^\theta \Div \vu_\e - \langle \vr_\e^\theta \Div \vu_\e \rangle \right)  }\,\dt\\
&=: I_7+ I_8.
\ea

By estimates \eqref{bd-vr}--\eqref{bd-vu} and the fact $\gamma>6$, we have
\ba\label{bd-vrvu-1}
&\|\vr_\e \vu_\e \|_{L^2(0,T;L^3(\O_\e))} \leq \|\vr_\e \|_{L^\infty(0,T;L^6(\O_\e))}  \|\vu_\e \|_{L^2(0,T;L^6(\O_\e))}\leq C,\\
&\|\vr_\e \vu_\e \|_{L^\infty(0,T;L^\frac{12}{7}(\O_\e))} \leq \| \sqrt \vr_\e \|_{L^\infty(0,T;L^{12}(\O_\e))}  \|\sqrt \vr_\e \vu_\e \|_{L^\infty(0,T;L^2(\O_\e))}\leq C.
\ea
This implies, by interpolation, that
\ba\label{bd-vrvu-2}
\|\vr_\e \vu_\e \|_{L^6(0,T;L^2(\O_\e))} \leq \|\vr_\e \vu_\e \|_{L^2(0,T;L^3(\O_\e))}^{\frac{1}{3}} \|\vr_\e \vu_\e\|_{L^\infty(0,T;L^\frac{12}{7}(\O_\e))}^{\frac 23}\leq C.
\ea

We apply Proposition \ref{prop-Bog} to get
\ba\label{est-I7}
|I_7| &\leq C \|\vr_\e \vu_\e \|_{L^6(0,T;L^2(\Omega_\e))}\norm{ \mathcal{B}_\e\left(\Div\big(\vr_\e^\theta \vu_\e \big) \right)}_{L^\frac{6}{5}(0,T;L^2(\Omega_\e))}\leq C \|\vr_\e^\theta \vu_\e \|_{L^\frac{6}{5}(0,T;L^2(\Omega_\e))}\\
& \leq C \norm{ \vu_\e}_{L^2(0,T;L^6(\Omega_\e))} \norm{\vr_\e^\theta }_{L^3(0,T;L^3(\Omega_\e))} \leq C \norm{ \nabla_x \vu_\e}_{L^2(0,T;L^2(\Omega_\e))} \norm{\vr_\e^\theta }_{L^3(0,T;L^3(\Omega_\e))} .
\ea

For $I_8$, by Proposition \ref{prop-Bog} and Sobolev embedding, together with \eqref{bd-vrvu-1} and \eqref{bd-vrvu-2}, we have
\ba\label{est-I8-1}
|I_8| &\leq C \norm{\vr_\e\vu_\e}_{L^6(0,T;L^2(\Omega_\e))} \norm{\mathcal{B}_\e\left(\vr_\e^\theta \Div \vu_\e - \langle \vr_\e^\theta \Div \vu_\e \rangle \right)}_{L^\frac{6}{5}(0,T;L^2(\Omega_\e))} \\
&\leq  C  \norm{\mathcal{B}_\e\left(\vr_\e^\theta \Div \vu_\e - \langle \vr_\e^\theta \Div \vu_\e \rangle \right)}_{L^\frac{6}{5}(0,T;W_0^{1,\frac{6}{5}}(\Omega_\e))}\\
&\leq C \norm{\vr_\e^\theta \Div \vu_\e}_{L^\frac{6}{5}(0,T;L^{\frac{6}{5}}(\Omega_\e))}\leq C   \norm{ \Div \vu_\e}_{L^2(0,T;L^2(\Omega_\e))} \norm{\vr_\e^\theta }_{L^3(0,T;L^3(\Omega_\e))}.
\ea

Thus, by \eqref{est-I7} and \eqref{est-I8-1}, we obtain
\ba\label{est-I8-2}
|I_5| \leq |I_7|+ |I_8| \leq C \norm{\vr_\e^\theta }_{L^3(0,T;L^3(\Omega_\e))} = C \norm{\vr_\e }_{L^{\frac{3\g}{2}}(0,T;L^{\frac{3\g}{2}}(\Omega_\e))}^{\frac{\g}{2}}.\nn
\ea

Finally, for $I_6$, it is direct to obtain
$$
|I_6|\leq \|\vr_\e \|_{L^\infty(0,T;L^2(\O_\e))} \|\vr_\e^{\th} \|_{L^\infty(0,T;L^2(\O_\e))} \leq C.
$$

Hence, summing up above estimates for $I_j, \ 1\leq j \leq 8,$ by passing $\psi \to 1$ in $L^\infty((0,T))$, we have
$$
\int_0^T \intOe{ p(\vr_\e)\vr_\e^{\frac{\g}{2}}}\,\dt \leq C + C \norm{\vr_\e }_{L^{\frac{3\g}{2}}(0,T;L^{\frac{3\g}{2}}(\Omega_\e))}^{\frac{\g}{2}}.
$$
Thus, by assumption on pressure in \eqref{pp1}, we obtain
$$
\norm{\vr_\e }_{L^{\frac{3\g}{2}}(0,T;L^{\frac{3\g}{2}}(\Omega_\e))}^{\frac{3\g}{2}} \leq C + C \int_0^T \intOe{ p(\vr_\e)\vr_\e^{\frac{\g}{2}}}\,\dt\leq C + C \norm{\vr_\e }_{L^{\frac{3\g}{2}}(0,T;L^{\frac{3\g}{2}}(\Omega_\e))}^{\frac{\g}{2}}.
$$
This implies
\ba\label{bd-vr3-1}
&\{\tilde\vr_\e\}_{\e>0} \ \mbox{uniformly bounded in}\ L^{\frac{3\gamma}{2}}((0,T)\times \Omega).
\ea

\medskip

Next, based on estimates obtained above in \eqref{bd-vr3-1}, we choose a bigger value $\th=\frac{2\g}{3}-1$ in \eqref{test-Be} to obtained our desired estimates \eqref{bd-vr2}--\eqref{bd-vr-p}. To this end, we estimate all $I_j$ with $\th:=\frac{2\g}{3}-1$.

For $I_1$,  estimate \eqref{bd-vr} and H\"older's inequality implies
$$
|I_1| \leq C \sup_{t\in [0,T]}\abs{\langle\vr_\e^\theta(t)\rangle}\intOe{ \abs{ p(\vr_\e(t)}}\leq C\norm{\vr_\e}_{L^\infty(0,T;L^\gamma(\Omega_\e))}^{\frac{5\g}{3}-1}\leq C.
$$

For $I_2$, by Proposition \ref{prop-Bog}, estimates \eqref{bd-vr}, \eqref{bd-vu}, \eqref{bd-vr3-1} and Sobolev embedding, we have
\ba
|I_2| \leq &  C \norm{\vr_\e}_{L^\infty(0,T;L^\gamma(\Omega_\e))}\norm{\bfu_\e}_{L^2(0,T;L^6(\Omega_\e))}\norm{ \mathcal{B}_\e\left(\vr_\e^\theta-\langle\vr_\e^\theta\rangle \right)}_{L^\infty(0,T;L^3(\Omega_\e))} \\
 \leq & C\norm{ \mathcal{B}_\e\left(\vr_\e^\theta-\langle\vr_\e^\theta\rangle \right)}_{L^\infty(0,T;W^{1,\frac{3}{2}}_0(\Omega_\e))}
 \leq  C \norm{ \vr_\e^\theta}_{L^\infty(0,T;L^\frac{3}{2}(\Omega_\e))}\leq C \norm{ \vr_\e}_{L^\infty(0,T;L^\g(\Omega_\e))}^{\frac{2\g}{3}-1}.\nn
\ea

For $I_3$,  by the fact $\gamma>6$, $\th=\frac{2\g}{3}-1$ and $\frac1\gamma+\frac13 + \frac{\th}{\g} = 1$, we obtain by using Proposition \ref{prop-Bog} and H\"older's inequality
\begin{align*}
\abs{I_3} &\leq C \norm{\vr_\e}_{L^\infty(0,T;L^\gamma(\Omega_\e))}\norm{\bfu_\e}_{L^2(0,T;L^6(\Omega_\e))}^2\norm{\Grad \mathcal{B}_\e\left(\vr_\e^\theta-\langle\vr_\e^\theta\rangle \right)}_{L^\infty(0,T;L^\frac{\g}{\th}(\Omega_\e))}\\
& \leq C \norm{\vr_\e^\theta }_{L^\infty(0,T;L^\frac{\g}{\th}(\Omega_\e))} \leq C \norm{\vr_\e }_{L^\infty(0,T;L^\g(\Omega_\e))}^{\th}.
\end{align*}

For $I_4$, by \eqref{bd-vu} and \eqref{bd-vr3-1}, and the fact $2\th < \frac{3\g}{2}$ for $\th = \frac{2\g}{3}-1$, we obtain
\ba
|I_4| &\leq C\norm{\nabla_x\bfu_\e}_{L^2(0,T;L^2(\Omega_\e))} \norm{\Grad \mathcal{B}_\e\left(\vr_\e^\theta-\langle\vr_\e^\theta\rangle \right)}_{L^2(0,T;L^2(\Omega_\e))} \\
& \leq C \norm{\vr_\e^\theta}_{L^2(0,T;L^2(\Omega_\e))} \leq C \norm{\vr_\e}_{L^{2\th}(0,T;L^{2\th}(\Omega_\e))}^\th \leq C.\nn
\ea

For $I_6$, it is direct to obtain
$$
|I_6|\leq \|\vr_\e \|_{L^\infty(0,T;L^2(\O_\e))} \|\vr_\e^{\th} \|_{L^2(0,T;L^2(\O_\e))} \leq C.
$$

\medskip

Now we estimate $I_5$ which is the most difficult one. By estimates in \eqref{bd-vr3-1}, we use Lemma \ref{lem-ren2} with $\beta=3\gamma/2$ to obtain
\be\label{eq-rth2}
\d_t (\vr^{\theta})+\Div\big(\vr^\theta \vu\big)+\big((\theta-1)\vr^\theta\big)\Div \vu=0\ \mbox{in}\ \mathcal{D}'\big((0,T)\times \R^3\big), \quad \theta:=\frac{2\gamma}{3}-1.\nn
\ee

We then split $I_5 = I_7 + I_8$ the same way as \eqref{est-B-dt-vr-1}:
\ba\label{def-I7-I8}
&I_7:= \int_0^T \intOe{ \psi\vr_\e\vu_\e \cdot  \mathcal{B}_\e\left(\Div\big(\vr_\e^\theta \vu_\e \big) \right) }\,\dt, \\
&I_8:= (\th-1)\int_0^T \intOe{ \psi\vr_\e\vu_\e \cdot  \mathcal{B}_\e\left(\vr_\e^\theta \Div \vu_\e - \langle \vr_\e^\theta \Div \vu_\e \rangle \right)  }\,\dt.\nn
\ea

We first estimate $\vr_\e \vu_\e$. By the fact
\ba
\left(2\left(\frac{5\g}{3}-1\right)\right)^{-1} + \frac{1}{2} = \left(\frac{10\g-6}{3}\right)^{-1} + \frac 12 = \frac{5\g}{10\g-6}, \nn
\ea
and H\"older's inequality, we have
\ba\label{est-I8-new-3}
&\left\|\vr_\e \vu_\e\right\|_{L^\frac{10\g-6}{3}\left(0,T; L^\frac{10\g-6}{5\g}(\O_\e)\right)} = \left\|\sqrt \vr_\e \sqrt\vr_\e\vu_\e\right\|_{L^\frac{10\g-6}{3}\left(0,T; L^\frac{10\g-6}{5\g}(\O_\e)\right)} \\
& \leq   \left\|\sqrt \vr_\e \right\|_{L^\frac{10\g-6}{3}\left(0,T; L^\frac{10\g-6}{3}(\O_\e)\right)} \left\| \sqrt\vr_\e\vu_\e\right\|_{L^\infty\left(0,T; L^2(\O_\e)\right)} \leq C \left\| \vr_\e \right\|_{L^\frac{5\g-3}{3}\left(0,T; L^\frac{5\g-3}{3}(\O_\e)\right)}^{\frac{1}{2}}.
 \ea
Similarly, by the fact
\ba
 \left(\frac{5\g}{3}-1\right)^{-1} + \frac{1}{2} = \frac{5\g+3}{10\g-6}, \quad  \left(\frac{5\g}{3}-1\right)^{-1} + \frac{1}{6} = \frac{5\g+15}{6(5\g-3)}, \nn
\ea
and H\"older's inequality, we have
\ba\label{est-I8-new-4}
&\left\|\vr_\e \vu_\e\right\|_{L^\frac{10\g-6}{5\g + 3}\left(0,T; L^\frac{6(5\g-3)}{5\g+15}(\O_\e)\right)}  \leq   \left\| \vr_\e \right\|_{L^\frac{5\g-3}{3}\left(0,T; L^\frac{5\g-3}{3}(\O_\e)\right)} \left\| \vu_\e\right\|_{L^2\left(0,T; L^6(\O_\e)\right)}\\
& \qquad \leq C \left\| \vr_\e \right\|_{L^\frac{5\g-3}{3}\left(0,T; L^\frac{5\g-3}{3}(\O_\e)\right)}.
 \ea
By \eqref{est-I8-new-3}, \eqref{est-I8-new-4}, and interpolations between Lebesgue spaces, we have, for any $\a\in [0,1]$ and $r_1, \ r_2$ such that
\be\label{def-r12-alpha}
\frac{1}{r_1} = (1-\a) \frac{3}{10\g-6} + \a  \frac{5\g + 3}{10\g-6}, \quad \frac{1}{r_2} = (1-\a) \frac{5\g}{10\g-6} + \a  \frac{5\g+15}{6(5\g-3)},
\ee
there holds
\ba\label{est-I8-new-5}
\left\|\vr_\e \vu_\e\right\|_{L^{r_1}\left(0,T; L^{r_2}(\O_\e)\right)}  \leq  C  \left\| \vr_\e \right\|_{L^\frac{5\g-3}{3}\left(0,T; L^\frac{5\g-3}{3}(\O_\e)\right)}^{\frac{(1+\a)}{2}}.\nn
 \ea
By choosing $\a=\frac{1}{5}$ in \eqref{def-r12-alpha}, we have
\be\label{r12-alpha=1/5}
\frac{1}{r_1} = \frac{\g + 3}{10\g-6}, \quad \frac{1}{r_2} = \frac{13\g+3}{6(5\g-3)}.\nn
\ee
This implies
\ba\label{est-I8-new-6}
\left\|\vr_\e \vu_\e\right\|_{L^{\frac{10\g-6}{\g + 3}}\left(0,T; L^{\frac{6(5\g-3)}{13\g+3}}(\O_\e)\right)}  \leq  C  \left\| \vr_\e \right\|_{L^\frac{5\g-3}{3}\left(0,T; L^\frac{5\g-3}{3}(\O_\e)\right)}^{\frac{3}{5}}.
 \ea

For $I_7$, by \eqref{est-I8-new-6} and the fact
$$
\frac{1}{2} + \frac{2\g-3}{5\g-3} = \frac{9\g-9}{10\g-6}, \quad \frac{1}{6} + \frac{2\g-3}{5\g-3} = \frac{17\g-21}{6(5\g-3)} <\frac{2}{3}
$$
and
$$
\frac{\g + 3}{10\g-6} + \frac{9\g-9}{10\g-6} =1, \quad \frac{13\g+3}{6(5\g-3)}+ \frac{17\g-21}{6(5\g-3)}=1,
$$
we can apply Proposition \ref{prop-Bog} to obtain
\ba\label{est-I7-new}
|I_7| &\leq  C \left\|\vr_\e \vu_\e\right\|_{L^{\frac{10\g-6}{\g + 3}}\left(0,T; L^{\frac{6(5\g-3)}{13\g+3}}(\O_\e)\right)}  \left\|\mathcal{B}_\e\left(\Div\big(\vr_\e^\theta \vu_\e \big) \right)\right\|_{L^\frac{10\g-6}{9\g-9}\left(0,T; L^\frac{6(5\g-3)}{17\g-21}(\O_\e)\right)}\\
&\leq C \left\|\vr_\e \vu_\e\right\|_{L^{\frac{10\g-6}{\g + 3}}\left(0,T; L^{\frac{6(5\g-3)}{13\g+3}}(\O_\e)\right)}  \left\|\vr_\e^\theta \vu_\e  \right\|_{L^\frac{10\g-6}{9\g-9}\left(0,T; L^\frac{6(5\g-3)}{17\g-21}(\O_\e)\right)}\\
&\leq  C  \left\| \vr_\e \right\|_{L^\frac{5\g-3}{3}\left(0,T; L^\frac{5\g-3}{3}(\O_\e)\right)}^{\frac{3}{5}} \left\|\vr_\e^\theta   \right\|_{L^\frac{5\g-3}{2\g-3}\left(0,T; L^\frac{5\g-3}{2\g-3}(\O_\e)\right)}\left\|\vu_\e   \right\|_{L^2\left(0,T; L^6(\O_\e)\right)} \\
&\leq  C  \left\| \vr_\e \right\|_{L^\frac{5\g-3}{3}\left(0,T; L^\frac{5\g-3}{3}(\O_\e)\right)}^{\frac{3}{5}+\th} \leq C \|\vr_\e\|_{L^\frac{5\g-3}{3}((0,T)\times \O_\e)}^{\frac{2\g}{3}-\frac{2}{5}}.\nn
\ea

It is left to estimate $I_8$. Since $\th=\frac{2\g}{3}-1$ and
$$
\frac{1}{2} + \frac{2\g-3}{5\g-3} = \frac{9\g-9}{10\g-6},
$$
H\"older's inequality implies
\ba\label{est-I8-new-1}
\left\|\vr_\e^\th \Div \vu_\e\right\|_{L^\frac{10\g-6}{9\g-9}((0,T)\times \O_\e)} \leq \left\|\Div \vu_\e\right\|_{L^2((0,T)\times \O_\e)}\|\vr_\e^\th\|_{L^\frac{5\g-3}{2\g-3}((0,T)\times \O_\e)} \leq C \|\vr_\e\|_{L^\frac{5\g-3}{3}((0,T)\times \O_\e)}^{\th}.\nn
\ea
Thus, by Proposition \ref{prop-Bog} and Sobolev embedding, we have
\ba\label{est-I8-new-2}
&\left\|\mathcal{B}_\e\left(\vr_\e^\th \Div \vu_\e- \langle \vr_\e^\theta \Div \vu_\e \rangle \right)\right\|_{L^\frac{10\g-6}{9\g-9}\left(0,T; L^\frac{6(5\g-3)}{17\g-21}(\O_\e)\right)}  \\
&\qquad \leq C \left \|\mathcal{B}_\e\left(\vr_\e^\th \Div \vu_\e - \langle \vr_\e^\theta \Div \vu_\e \rangle \right)\right\|_{L^\frac{10\g-6}{9\g-9}\left(0,T;W^{1,\frac{10\g-6}{9\g-9}}_0(\O_\e)\right)}\\
&\qquad \leq C \left\|\vr_\e^\th \Div \vu_\e\right\|_{L^\frac{10\g-6}{9\g-9}((0,T)\times \O_\e)} \leq C \|\vr_\e\|_{L^\frac{5\g-3}{3}((0,T)\times \O_\e)}^{\th}.
\ea
By estimates \eqref{est-I8-new-2} and \eqref{est-I8-new-6}, we obtain
\ba
|I_8| &\leq C \left\|\mathcal{B}_\e\left(\vr_\e^\th \Div \vu_\e- \langle \vr_\e^\theta \Div \vu_\e \rangle \right)\right\|_{L^\frac{10\g-6}{9\g-9}\left(0,T; L^\frac{6(5\g-3)}{17\g-21}(\O_\e)\right)} \left\|\vr_\e \vu_\e\right\|_{L^{\frac{10\g-6}{\g + 3}}\left(0,T; L^{\frac{6(5\g-3)}{13\g+3}}(\O_\e)\right)}\\
&\leq C \|\vr_\e\|_{L^\frac{5\g-3}{3}((0,T)\times \O_\e)}^{\th}  \|\vr_\e\|_{L^\frac{5\g-3}{3}((0,T)\times \O_\e)}^{\frac{3}{5}} \leq C \|\vr_\e\|_{L^\frac{5\g-3}{3}((0,T)\times \O_\e)}^{\frac{2\g}{3}-\frac{2}{5}}.\nn
\ea

Hence, summing up above estimates for $I_j, \ 1\leq j \leq 8,$ by passing $\psi \to 1$ in $L^\infty((0,T))$, we have
$$
\int_0^T \intOe{ p(\vr_\e)\vr_\e^{\frac{2\g}{3}-1}}\,\dt \leq C + C \|\vr_\e\|_{L^\frac{5\g-3}{3}((0,T)\times \O_\e)}^{\frac{2\g}{3}-\frac{2}{5}}.
$$
Thus, by assumption on pressure in \eqref{pp1}, we have
$$
\|\vr_\e\|_{L^\frac{5\g-3}{3}((0,T)\times \O_\e)}^{\frac{5\g-3}{3}} \leq C + C \int_0^T \intOe{ p(\vr_\e)\vr_\e^{\frac{2\g}{3}-1}}\,\dt \leq C + C \|\vr_\e\|_{L^\frac{5\g-3}{3}((0,T)\times \O_\e)}^{\frac{2\g}{3}-\frac{2}{5}}.
$$
This implies our desired uniform estimates in \eqref{bd-vr2} and \eqref{bd-vr-p}.

\subsection{Equations in homogeneous domains}

In this section, we derive the equations in $\tilde\vr_\e,\tilde\vu_\e$ in $((0,T)\times \Omega)$.

\subsubsection{Continuity equation}
For the continuity equation, we have the following Proposition:
\begin{Proposition}\label{prop-con}
Under the assumption in Theorem 1.1, the extension functions  $\tilde\vr_\e,\tilde \vu_\e$  satisfy
 \be\label{eq-tvr}
 \d_t \tilde \vr_\e+\Div(\tilde\vr_\e\tilde \vu_\e)=0,\quad \mbox{in}\ \mathcal{D}'((0,T)\times \R^3).
 \ee
\end{Proposition}

\begin{proof}[Proof of Proposition \ref{prop-con}]  It is sufficient to prove
 \be\label{eq-tvr-wk}
 \int_0^T \intO{ \tilde \vr_\e \d_t \psi+\tilde\vr_\e\tilde \vu_\e\cdot \nabla_x \psi}\dt=0 ,\quad \mbox{for any}\ \psi\in C_c^\infty((0,T)\times \R^3).\nn
 \ee
Let $\{\phi_n\}_{n\geq 1} \subset C_c^\infty(\Omega_\e)$ such that $0\leq \phi_n\leq 1$,  $|\nabla_x \phi_n|\leq 4n$ and
\ba\label{def-phin}
\phi_n=1 \ \mbox{on}\  \{x\ |\ {\rm dist}(x,\d\Omega_\e)\geq n^{-1}\},\quad \phi_n=0 \ \mbox{on}\  \left\{x\ |\ {\rm dist}(x,\d\Omega_\e)\leq (2n)^{-1}\right\}.\nn
\ea
Then
\ba\label{vr-n}
\int_0^T \intRt{ \tilde \vr_\e \d_t \psi+\tilde\vr_\e\tilde \vu_\e\cdot \nabla_x \psi}\dt=\int_0^T \intOe{  \vr_\e \d_t (\psi\phi_n)+\vr_\e \vu_\e\cdot \nabla_x (\psi\phi_n)}\dt\\
+\int_0^T \intOe{  \vr_\e \d_t \psi(1-\phi_n)+\vr_\e  \vu_\e\cdot \nabla_x \psi(1-\phi_n)-\vr_\e \vu_\e\cdot  \psi \nabla_x\phi_n}\dt.
\ea

By the estimates obtained in (\ref{bd-vr})-(\ref{bd-vu}) and Sobolev embedding, we have
$$
\lim_{n\to \infty}\int_0^T \intOe{  \vr_\e \d_t \psi(1-\phi_n)+\vr_\e  \vu_\e\cdot \nabla_x \psi(1-\phi_n)}\dt=0,
$$
where we used the fact
$$
1-\phi_n \to 0\ \mbox{in}\  L^q (\Omega_\e),\ \mbox{for any} \ 1<q<\infty,\ \mbox{as} \ n\to \infty.
$$
By virtue of (\ref{bd-vu}), we have
$$
{\rm dist}(x,\d\Omega_\e)^{-1} \vu_\e \in L^2((0,T)\times \Omega_\e).
$$
Then by (\ref{bd-vr}) and the fact
$$
{\rm dist}(x,\d\Omega_\e) \nabla_x \phi_n \to 0\ \mbox{in}\  L^q (\Omega_\e),\ \mbox{for any} \ 1<q<\infty,\ \mbox{as} \ n\to \infty£¬
$$
we have
$$
\lim_{n\to \infty}\int_0^T \intOe{  \vr_\e  \vu_\e\cdot  \psi\nabla_x \phi_n}\dt=0.
$$
We complete the proof by passing $n\to \infty$ in (\ref{vr-n})

\end{proof}

\subsubsection{Momentum equation}

By Remark \ref{rem-tech}, the condition (\ref{ass-tec}) implies that there exits $5/2<q<3$ such that \eqref{ass-tec1}, \eqref{ass-tec2} and \eqref{ass-tec3} are all satisfied. We now prove the following proposition
\begin{Proposition}\label{prop-mom}
Under the assumption in Theorem 1.1, there holds the equation
 \be\label{eq-tvu}
 \d_t (\tilde \vr_\e\tilde\vu_\e)+\Div(\tilde\vr_\e\tilde \vu_\e\otimes\tilde \vu_\e)+\nabla_x p(\tilde\vr_\e)=\Div{\mathbb S}(\Grad \tilde\vu_\e)+ \tilde \vr_\e \ff + F_\e,\quad \mbox{in}\ \mathcal{D}'((0,T)\times \Omega;\R^3),
 \ee
where $F_\e\in \mathcal{D}'((0,T)\times \Omega;\R^3)$ satisfies
\be\label{Fe}
|\langle F_\e,\varphi \rangle|\leq C  \e^{\sigma} \left(|\d_t \varphi|_{L^2(0,T;L^2(\Omega;\R^3))}+|\nabla_x\varphi|_{L^r(0,T;L^3(\Omega;\R^9))}+|\varphi|_{L^r0,T;L^r(\Omega;\R^3))}\right),
\ee
for any $\varphi \in C^\infty_c((0,T)\times \Omega;\R^3)$ and some constant $C$, some $1<r<\infty$, with $\sigma:=(3-q)\alpha-3>0$. Here $q\in (5/2,3)$ satisfying  \eqref{ass-tec1}, \eqref{ass-tec2} and \eqref{ass-tec3}, $C$ and $r$ are independent of $\e$.

\end{Proposition}

\begin{proof}[Proof of Proposition \ref{prop-mom}] Let $\varphi\in C_c^\infty((0,T)\times \Omega;\R^3)$ be any test function. It is sufficient to show:
\ba
I^\e:&=\int_{0}^T  \intO{\tilde \vr_\e\tilde \vu_\e \d_t \varphi+\tilde\vr_\e\tilde \vu_\e\otimes\tilde \vu_\e:\nabla_x \varphi+ p(\tilde\vr_\e)\Div \varphi-{\mathbb S}(\Grad \tilde\vu_\e):\nabla_x \varphi + \tvr{\bf f} \cdot \varphi}\dt\\
 &\leq C \ \e^{\sigma} \left(|\d_t \varphi|_{L^2(0,T;L^2(\Omega;\R^3))}+|\nabla_x\varphi|_{L^r(0,T;L^3(\Omega;\R^9))}+|\varphi|_{L^r(0,T;L^r(\Omega;\R^3))}\right),\ \mbox{for some $r<\infty$}.\nn
\ea

 The condition (\ref{ass-tec1}) makes sure that we can apply Lemma \ref{g-exi} to find $\{g_\e\}_{\e>0}\subset W_0^{1,q}(\Omega)$ such that (\ref{g-fam}) and (\ref{g-est}) are satisfied; the conditions (\ref{ass-tec2}) as well as (\ref{ass-tec3}) make sure all the terms appeared in the following equation make sense:
\ba\label{Ie}
I^\e& = \int_{0}^T  \intOe{ \vr_\e\vu_\e \d_t (g_\e \varphi)+\vr_\e \vu_\e\otimes \vu_\e:\nabla_x (g_\e\varphi) + p(\vr_\e)\Div (g_\e\varphi)\\
&\qquad -{\mathbb S}(\Grad \vu_\e):\nabla_x (g_\e\varphi) + \vr_\e \ff \cdot (g_\e \varphi)}\dt +\sum_{j=1}^5 I_j= \sum_{j=1}^5 I_j,\nn
\ea
where
\ba\label{Ij-1-5}
I_1&:=\int_{0}^T  \intO{ \tilde\vr_\e\tilde\vu_\e (1-g_\e) \d_t \varphi}\dt,\\
I_2&:=\int_{0}^T  \intO{ \tilde \vr_\e \tilde \vu_\e\otimes \tilde \vu_\e:(1-g_\e)\nabla_x \varphi -\tilde \vr_\e \tilde \vu_\e\otimes \tilde \vu_\e:(\nabla_x g_\e\otimes  \varphi) }\dt\\
I_3&:=\int_{0}^T  \intO{  p(\tilde \vr_\e)(1-g_\e)\Div\varphi -p(\tilde \vr_\e)\nabla_x g_\e\cdot \varphi}\dt,\\
I_4&:=\int_{0}^T  \intO{ {\mathbb S}(\Grad \tilde \vu_\e):(1-g_\e)\nabla_x \varphi +{\mathbb S}(\Grad \tilde \vu_\e):(\nabla_x g_\e\otimes  \varphi) }\dt,\\
I_5&:=\int_{0}^T  \intO{ \tilde\vr_\e\tilde\ff (1-g_\e)  \varphi}\dt.\nn
\ea

We now estimate $I_j$ one by one. By (\ref{g-est}) in Lemma \ref{g-exi} and Sobolev embedding, we have
\be\label{g-est2}
\|g_\e-1\|_{L^{q^*}}\leq C \ \e^\sigma,\quad \frac{1}{q^*}=\frac{1}{q}-\frac{1}{3}.
\ee

We further observe that the condition (\ref{ass-tec}) is sufficient to make the inequality (\ref{ass-tec2}) and (\ref{ass-tec3}) be strict. Precisely, under condition (\ref{ass-tec}), there exists $5/2<q<3$ such that (\ref{ass-tec1}) and the following two inequalities are satisfied:
\be\label{ass-tec4}
\left(\frac{5}{3}-\frac{1}{\gamma}\right)^{-1}+\frac{1}{q} < 1,\quad \frac{1}{\gamma}+\frac{1}{3}+\frac{1}{q}< 1.
\ee

In the rest of this proof, we use a simpler notation $L^{q}L^r$ to denote spaces $L^q(0,T; L^r(\O))$ or $L^q(0,T; L^r(\O;\R^3))$ or  $L^q(0,T; L^r(\O;\R^9))$.

\medskip

We now consider $I_1$. By (\ref{g-est2}) and (\ref{ass-tec4}), there holds
\be\label{ass-tec5}
\frac{1}{\gamma}+\frac{1}{6}+\frac{1}{q^*}< \frac{1}{2}.
\ee

 By the uniform estimates in (\ref{bd-vr})-(\ref{bd-vr2}) and H\"older's inequality,  we obtain
\[
I_1 \leq   \|\tilde\vr_\e\|_{L^\infty L^\gamma}\|\tilde u_\e\|_{L^2L^6} \|1-g_\e\|_{L^{q^*}}\|\d_t\varphi\|_{L^2 L^2} \leq C\ \e^{\sigma} \|\d_t\varphi\|_{L^2 L^2}.
\]

\medskip

For $I_2$,  by (\ref{g-est2}) and (\ref{ass-tec4}), using Sobolev embedding and H\"older inequality, we have
\be\label{bd-vru1}
\{\tilde \vr_\e |\tilde\vu_\e|^2\}_{\e>0} \ \mbox{uniformly bounded in}\ L^\infty L^1 \cap L^1L^{\frac{3\gamma}{3+\gamma}}.\nn
\ee
By interpolation and the inequalities in (\ref{ass-tec4}) and (\ref{ass-tec5}), we have
\be\label{bd-vru2}
\{\tilde \vr_\e |\tilde\vu_\e|^2\}_{\e>0} \ \mbox{uniformly bounded in}\ L^{q_1} L^{q_2}\nn
\ee
for some $1<q_1<\infty$ and $1<q_2<3\gamma/(3+\gamma)$ satisfying
\be\label{ass-tec6}
\frac{1}{q_2}+\frac{1}{q}< 1,\quad\frac{1}{q_2}+\frac{1}{q^*}< \frac{2}{3}.\nn
\ee
Then H\"older's inequality implies
\[
I_2\leq C \|\tilde\vr_\e |\tilde\vu_\e|^2\|_{ L^{q_1} L^{q_2}}  \|1-g_\e\|_{L^{q^*}}\|\nabla_x\varphi\|_{L^{r_1} L^{3}} + C \|\tilde\vr_\e |\tilde\vu_\e|^2\|_{ L^{q_1} L^{q_2}} \|1-g_\e\|_{W^{1,q}}\|\varphi\|_{L^{ r_1} L^{r_2}}
\]
\[
\leq C \e^{\sigma}(\|\nabla_x\varphi\|_{L^{r_1} L^{3}}+\|\varphi\|_{L^{ r_1} L^{ r_2}}),
\]
where
$$
\frac{1}{ r_1}:=1-\frac{1}{q_1}>0,\quad \frac{1}{ r_2}:=1-\frac{1}{q_2}-\frac{1}{q}>0.
$$

\medskip

For $I_3$, by (\ref{bd-vr2}), similar argument as the estimate for $I_2$ gives
\ba
I_3& \leq C   \|p(\tilde\vr_\e)\|_{ L^{\frac{5}{3}-\frac{1}{\gamma}} L^{\frac{5}{3}-\frac{1}{\gamma}}}  \|1-g_\e\|_{L^{q^*}}\|\nabla_x\varphi\|_{L^{r_3} L^{r_4}} + C  \|p(\tilde\vr_\e)\|_{ L^{\frac{5}{3}-\frac{1}{\gamma}} L^{\frac{5}{3}-\frac{1}{\gamma}}} \|1-g_\e\|_{W^{1,q}}\|\varphi\|_{L^{r_5} L^{\tilde r_6}}\\
&\leq C \e^{\sigma}(\|\nabla_x\varphi\|_{L^{ r_3} L^{r_4}}+\|\varphi\|_{L^{r_5} L^{r_6}}),\nn
\ea
where
\[
\frac{1}{r_3}:=1-\left(\frac{5}{3}-\frac{1}{\gamma}\right)^{-1}>\frac{1}{3}, \quad \frac{1}{r_4}:=1-\left(\frac{5}{3}-\frac{1}{\gamma}\right)^{-1}-\frac{1}{q^*}>\frac{1}{3},
\]
\[
\frac{1}{r_5}:=1-\left(\frac{5}{3}-\frac{1}{\gamma}\right)^{-1}>\frac{1}{3}, \quad \frac{1}{r_6}:=1-\left(\frac{5}{3}-\frac{1}{\gamma}\right)^{-1}-\frac{1}{q}>0.
\]

\medskip

Similarly, we have for $I_4$, we have
\[
I_4 \leq C \e^{\sigma}(\|\nabla_x\varphi\|_{L^{2} L^{r_7}}+\|\varphi\|_{L^{2} L^{r_8}}),\quad
\frac{1}{r_7}:=1-\frac{1}{2}-\frac{1}{q^*}>\frac{1}{3}+\frac{1}{10}, \ \frac{1}{r_8}:=1-\frac{1}{2}-\frac{1}{q}>\frac{1}{10}.
\]

Finally for $I_5:$
\[
I_5 \leq C \e^{\sigma} \|\varphi\|_{L^{2} L^{r_9}},\quad
\frac{1}{r_9}:=1-\frac{1}{q^*}- \frac{1}{\g}> 1-\frac{2}{5}-\frac{1}{6}>\frac{2}{5}.
\]

We thus complete the proof by taking
$$
r:=\max\{r_j\ |\ 1\leq j\leq 9\} \in (1,\infty).
$$

\end{proof}

\subsection{Passing to the limit}\label{sec:limitpass}

We need to show the weak convergence of the nonlinear terms in the sense of contribution:
\ba\label{weak-lim-non}
\tilde \vr_\e \tilde \vu_\e\to \vr\vu,\quad
\tilde \vr_\e \tilde \vu_\e\otimes \tilde \vu_\e \to \vr\vu\otimes\vu,\quad
p(\tilde\vr_\e)\to p(\vr).\nn
\ea

\subsubsection{Time derivative estimates and weak limits}

We would like to prove
\[
\tilde \vr_\e \tilde \vu_\e\to \vr\vu,\quad
\tilde \vr_\e \tilde \vu_\e\otimes \tilde \vu_\e \to \vr\vu\otimes\vu
\]
in the sense of distribution. A key point is that we can obtain some uniform estimate for $\d_t\tilde \vr$ and $\d_t(\tilde \vr\tilde\vu)$ by Proposition \ref{prop-con} and Proposition \ref{prop-mom} we just proved. This allows us to use the Aubin-Lions type argument to prove the weak limit of the product is the product of the weak limits.

\medskip

By Proposition \ref{prop-con} and the uniform estimates (\ref{bd-vr})-(\ref{bd-vu}), we have
\be\label{dt-vr}
\{\d_t \tilde \vr_\e\}_{\e>0}\ \mbox{uniformly bounded in}\ L^2(0,T;W^{-1,\frac{6\gamma}{6+\gamma}(\Omega)}).
\ee
By Aubin-Lions type argument (see for example \cite{Simon}), or directly using Lemma 5.1 in \cite{LI4}, we have
\be\label{wl-pd1}
\tilde \vr_\e \tilde \vu_\e\to \vr\vu \quad \mbox{in}\quad \mathcal{D}'((0,T)\times \Omega).
\ee

By \eqref{bd-vr} and \eqref{bd-vru}, there holds
\be\label{uniform-est-vrvu}
\{(\tilde \vr_\e \tilde\vu_\e)\}_{\e>0}\ \mbox{uniformly bounded in} \ L^\infty(0,T;L^\frac{2\gamma}{1+\gamma}(\Omega;\R^3)),\nn
\ee
together with \eqref{dt-vr},  we have
\be\label{wl-pd1-2}
  \tilde \vr_\e \to \vr \quad \mbox{in}\quad C_w([0,T];L^\g(\Omega;\R^3)),\quad    \tilde \vr_\e \tilde \vu_\e\to \vr\vu \quad \mbox{in}\quad C_w([0,T];L^\frac{2\gamma}{1+\gamma}(\Omega;\R^3)).\nn
\ee

\medskip

By Proposition \ref{prop-mom} and the uniform estimates (\ref{bd-vr})-(\ref{bd-vu}), have
\be\label{dt-vu1}
 \tilde \vr_\e \tilde\vu_\e=(\tilde \vr_\e \tilde\vu_\e)^{(1)}+\e^\sigma (\tilde \vr_\e \tilde\vu_\e)^{(2)},
\ee
where
\be\label{dt-vu2}
\{\d_t (\tilde \vr_\e \tilde\vu_\e)^{(1)}\}_{\e>0}\ \mbox{uniformly bounded in}\ L^{1}(0,T;W^{-1,{1}}(\Omega;\R^3)),
\ee
\be\label{dt-vu3}
\{(\tilde \vr_\e \tilde\vu_\e)^{(2)}\}_{\e>0}\ \mbox{uniformly bounded in} \ L^2((0,T)\times \Omega;\R^3).
\ee
By observing
$$
\{(\tilde \vr_\e \tilde\vu_\e)\}_{\e>0}\ \mbox{uniformly bounded in} \ L^2(0,T;L^\frac{6\gamma}{6+\gamma}(\Omega;\R^3))\subset  \ L^2((0,T)\times \Omega;\R^3),
$$
we have
\be\label{dt-vu4}
\{(\tilde \vr_\e \tilde\vu_\e)^{(1)}\}_{\e>0}\ \mbox{uniformly bounded in} \ L^2((0,T)\times \Omega;\R^3).
\ee
By (\ref{dt-vu2}) and (\ref{dt-vu4}), together with Aubin--Lions type argument (or using directly Lemma 5.1 in \cite{LI4}), we have
\be\label{wl-pd2}
(\tilde \vr_\e \tilde\vu_\e)^{(1)} \otimes\tilde\vu_\e \to \overline{(\tilde \vr_\e \tilde\vu_\e)^{(1)}}\otimes \vu \quad \mbox{in}\quad \mathcal{D}'((0,T)\times \Omega).\nonumber
\ee
By (\ref{wl-pd1}), (\ref{dt-vu1}), (\ref{dt-vu3}), the weak limit satisfies
\[
\overline{(\tilde \vr_\e \tilde\vu_\e)^{(1)}} = \overline{(\tilde \vr_\e \tilde\vu_\e)}=\vr\vu.
\]
Finally, by observing
$$
\e^\sigma \|(\tilde \vr_\e \tilde\vu_\e)^{(2)}\otimes \vu_\e\|_{L^1((0,T)\times \Omega;\R^3\times \R^3)} \leq C \e^\sigma \to 0,\ \mbox{ as }\ \e \ \to\  0,
$$
we obtain
\be\label{wl-pd3}
\tilde \vr_\e \tilde\vu_\e \otimes\tilde\vu_\e \to  \vr \vu \otimes \vu \quad \mbox{in}\quad \mathcal{D}'((0,T)\times \Omega).
\ee

Thus, by passing $\e \to 0$ in \eqref{eq-tvr} and \eqref{eq-tvu}, we obtain the equations in $(\vr, \vu)$:
 \be\label{eq-vr-lim1}
 \d_t \vr+\Div(\vr \vu)=0,\quad \mbox{in}\ \mathcal{D}'((0,T)\times \R^3),
 \ee
\be\label{eq-vu-lim1}
 \d_t (\vr \vu)+\Div( \vr \vu \otimes  \vu )+\nabla_x \overline{p(\vr)}=\Div{\mathbb S}(\Grad \vu)+ \vr \ff, \quad \mbox{in}\ \mathcal{D}'((0,T)\times \Omega;\R^3).
 \ee
Here in \eqref{eq-vu-lim1}, $\overline{p(\vr)}$ denotes the weak convergence of $p(\tilde \vr_\e)$. Moreover, by Lemma \ref{lem-ren2}, there holds
\be\label{reno-vr-lim2}
\d_t b(\vr)+\Div\big(b(\vr)\vu\big)+\big(b'(\vr)\vr-b(\vr)\big)\Div \vu=0\ \mbox{in}\ \mathcal{D}'\big((0,T)\times \R^3\big),
\ee
for any $b$ fulfilling properties stated in Lemma \ref{lem-ren2}.

\subsubsection{Strong convergence of the density}

The next step is to show
\be\label{weak-limit-pressure}
\overline{p(\vr)} = p(\vr),
\ee
which is a consequence of the strong convergence $\tilde \vr_\e \to \vr$ a.e. in $(0,T)\times \O$.

As in the existence theory of weak solutions for compressible Navier--Stokes equations, the strong convergence for the density approximate solution family is a tricky part. P.-L. Lions \cite{LI4} introduced the so-called effective viscous flux which enjoys an additional compactness, and this property plays a crucial role in the existence theory of weak solutions for the compressible Navier-Stokes equations. This is the following lemma:

\begin{Lemma}\label{lem:flux}
Up to a substraction of subsequence, there holds for any $\psi\in C_c^\infty(\Omega)$ and any $\phi \in C_c^\infty((0,T))$:
\ba\label{flux}
\lim_{\e\to 0}\int_0^T\intO{\phi(t)\psi(x) \left(p(\tilde\vr_\e)-(\frac{4\mu}{3}+\eta)\Div \tilde\vu_\e\right)\tilde\vr_\e}\,\dt\\
=\int_0^T\intO{\phi(t)\psi(x) \left(\overline{p(\vr)}-(\frac{4\mu}{3}+\eta)\Div  \vu\right)\vr}\,\dt.\nn
\ea

\end{Lemma}
\begin{proof}[Proof of  Lemma \ref{lem:flux}]
The proof of Lemma \ref{lem:flux} is quite tedious but nowadays well understood. In this proof, the notations $\nabla$ and $\Delta$ are all associated with spatial variable $x$.

The main idea is to employ the following test functions in the weak formulations of momentum equations:
\be\label{def-test-flux}
\phi \psi \nabla \Delta^{-1} \tilde\vr_\e, \quad  \phi \psi \nabla \Delta^{-1}\vr,
\ee
where $\Delta^{-1}$ is the Fourier multiplier on $ \R^3$ with symbol $-\frac{1}{|\xi|^2}$. In \eqref{def-test-flux},  $\tilde \vr_\e$ and $\vr$ are treated as functions in $\R^3$ with respect to spatial variable, where there holds $\tilde\vr_\e = \vr =0$ on $\O^c$.

Notice that
$$
\nabla \nabla \Delta^{-1}=\left(\mathcal{R}_{i,j}\right)_{1\leq i,j\leq 3}
$$
are the classical Riesz operators.  We refer to Section 1.3.7.2 in \cite{N-book} for the concepts of Fourier multiplier and Riesz operators. We recall the classical property for Riesz operators: for any $f\in L^q(\R^3),~1<r<\infty$:
\be\label{est-test-flux1}
\|\nabla\nabla \Delta^{-1} f\|_{L^r( \R^3; \R^9)}\leq C \|f\|_{L^r( \R^3)}.
\ee
By embedding theorems in homogeneous Sobolev spaces (see Theorem 1.55 and Theorem 1.57 in \cite{N-book} or Theorem 10.25 and Theorem 10.26 in \cite{F-N-book}), we have for any  $f\in L^r( \R^3)$ with ${\rm supp}\, f\subset \overline\O$:
\ba\label{est-test-flux2}
&\|\nabla \Delta^{-1}f\|_{L^{r^*}( \R^3; \R^3)}\leq C \| f \|_{L^r( \R^3)},\quad \frac{1}{r^*}=\frac{1}{r}-\frac{1}{3}, \ \mbox{if $1<r<3$},\\
&\|\nabla \Delta^{-1} f \|_{L^{r^*}( \R^3; \R^3)}\leq C \| f \|_{L^r( \R^3)},\quad \mbox{for any} \ r^*<\infty,  \mbox{if $r\geq3$}.
\ea
Then by the uniform estimate for $\tilde \vr_\e$ and its weak limit $\vr$ in \eqref{bd-vr} and \eqref{weak-limit}, we have
\ba\label{est-test-flux3}
\|\nabla  \nabla \Delta^{-1} \tilde\vr_\e\|_{L^\infty(0,T;L^{\g}(\O; \R^9))} + \|\nabla  \nabla \Delta^{-1}\vr \|_{L^\infty(0,T;L^{\g}(\O; \R^9))} \leq C.
\ea
Moreover, under the restriction $\g>6$, we have for any $r\in (1,\infty)$:
\ba\label{est-test-flux3-0}
\|\nabla \Delta^{-1} \tilde\vr_\e \|_{L^\infty(0,T; L^r(\O; \R^3))} + \| \nabla \Delta^{-1} \vr \|_{L^\infty(0,T; L^r(\O; \R^3))}\leq C.
\ea

By equations \eqref{eq-tvr} and \eqref{eq-vr-lim1}, there holds in $\mathcal{D}'((0,T)\times \R^3)$:
\ba\label{dt-tvr-vr}
&\d_t \left( \nabla \Delta^{-1}\tilde\vr_\e  \right) = \nabla \Delta^{-1}( \d_t \tilde\vr_\e)= -  \nabla \Delta^{-1}(\Div (\tilde\vr_\e \tilde \vu_\e)),\\
&\d_t \left( \nabla \Delta^{-1} \vr  \right) =  \nabla \Delta^{-1}( \d_t \vr ) = - \nabla \Delta^{-1}(\Div (\vr \vu)).
\ea
Then, by uniform estimates in \eqref{bd-vu} and the results in \eqref{est-test-flux1}--\eqref{est-test-flux2}, we have from \eqref{dt-tvr-vr} that
\ba\label{est-dt-vr}
&\left\|\d_t \left( \nabla \Delta^{-1}\tilde\vr_\e  \right)\right\|_{{L^2(0,T;L^{\frac{6\g}{6+\g}}(\O; \R^3))}} \leq C \left\|\tilde\vr_\e \right\|_{L^\infty(0,T;L^{\g}(\O))} \left\|\tilde\vu_\e \right\|_{L^2(0,T;W_0^{1,2}(\O;\R^3))}\leq C, \\
&\left\|\d_t \left( \nabla \Delta^{-1} \vr  \right)\right\|_{{L^2(0,T;L^{\frac{6\g}{6+\g}}(\O; \R^3))}} \leq C \left\| \vr  \right\|_{L^\infty(0,T;L^{\g}(\O))} \left\| \vu  \right\|_{L^2(0,T;W_0^{1,2}(\O;\R^3))}\leq C¡£
\ea

By \eqref{est-test-flux3}--\eqref{est-dt-vr} and Aubin-Lions-Simon Theorem (see \cite{Simon}), we have that the family  $\{\nabla \Delta^{-1}\tilde\vr_\e \}_{\e>0}$ is strongly precompact in $L^\infty(0,T; L^r(\O;\R^3))$ for any $r\in (1,\infty)$. Thus, up to a substraction of a subsequence,
\be\label{con-st-Avr}
\nabla \Delta^{-1}\tilde\vr_\e \to \nabla \Delta^{-1} \vr \quad \mbox{strongly in} \ C_w([0,T]; L^r(\O;\R^3)) \ \mbox{for any $r\in (1,\infty)$}.
\ee

For the residual term $F_\e$ in \eqref{eq-tvu}, by \eqref{Fe},  \eqref{est-test-flux3}, \eqref{est-test-flux3-0} and \eqref{est-dt-vr}, we have
\ba\label{est-test-flux4}
|\langle F_\e, \phi \psi \nabla \Delta^{-1}\tilde\vr_\e\rangle| \leq C \e^{\sigma}
\ea
which goes to zero as $\e\to 0$. Here we used the fact $6\g/(6+\g)>3>2$ under restriction $\g>6$.

\smallskip

Now we take $\phi \psi \nabla \Delta^{-1}(1_{\Omega}\tilde\vr_\e)$ as a test functions in the weak formulation of \eqref{eq-tvu} and pass $\e\to 0$. Then we take $\phi \psi \nabla \Delta^{-1}(1_{\Omega}\vr)$ as a test functions in the weak formulation of $\eqref{eq-vu-lim1}$. By comparing the results of theses two operations, through long but straightforward calculations, using the convergence results in \eqref{con-st-Avr} and \eqref{est-test-flux4}, we finally obtain
\ba\label{flux-l1}
I:&=\lim_{\e\to 0}\int_0^T\intO{\phi \psi \left(p(\tilde\vr_\e)-(\frac{4\mu}{3}+\eta)\Div \tilde\vu_\e\right)\tilde\vr_\e}\,\dt \\
&\qquad - \int_0^T \intO{\phi \psi\left(\overline{p(\vr)}-(\frac{4\mu}{3}+\eta)\Div \vu\right)\vr}\,\dt\\
&=\lim_{\e\to 0}\int_0^T \phi \intO{ \psi \tilde\vr_\e\tilde\vu_\e^i\tilde\vu_\e^j  \mathcal{R}_{i,j}( \tilde\vr_\e)}\,\dt-\int_0^T \phi \intO{ \psi  \vr \vu^i\vu^j \mathcal{R}_{i,j}( \vr)}\,\dt.
\ea
The last quantity in \eqref{flux-l1} is indeed zero. This follows by a div-curl type, see \cite[Section 3.4]{FNP}.

\end{proof}

\medskip





Now we are ready to show the strong convergence of $\{\tilde \vr_\e\}_{\e>0}$. First of all, we have
$ \big(\frac{5\g}{3}-1 \big)- (\g+1) = \frac{2\g}{3} -2 >0.$ Then by \eqref{bd-vr2} and \eqref{bd-vr-p}, we have
$$
p(\tilde \vr_\e)\tilde \vr_\e \to \overline {p(\vr)\vr} \quad \mbox{weakly in}\quad L^{\left(\frac{5\g}{3}-1\right)/(\g+1)}((0,T)\times \O)).
$$

Taking $b(s)=s\log s$ in the renormalized equations \eqref{eq-vr-b} and \eqref{reno-vr-lim2} implies
\ba\label{eq:slogs}
&\d_t(\tilde\vr_\e\log \tilde\vr_\e) + \Div \big((\tilde\vr_\e\log \tilde\vr_\e) \tilde\vu_\e\big)+\tilde\vr_\e \Div  \tilde\vu_\e=0,\ \mbox{in}\ \mathcal{D}'((0,T)\times \R^3)).\\
& \d_t(\vr\log \vr) +  \Div \big((\vr\log \vr) \vu\big)+\vr \Div  \vu=0,\ \mbox{in}\ \mathcal{D}'((0,T)\times \R^3)).
\ea
Passing $\e\to 0$ in the first equation of \eqref{eq:slogs} gives
\be\label{eq:slogs1}
 \d_t(\overline{\vr\log \vr})  + \Div \big(\overline{(\vr\log \vr)\vu}\big)+\overline{\vr \Div  \vu}=0, \ \mbox{in}\ \mathcal{D}'((0,T)\times \R^3)).
\ee
Then, by using a test function sequence $\{\phi_n(t)\}_n \subset C_c^\infty(0,T)$ which convergence to $1$ strongly in $L^2(0,T)$ in the weak formulation of $\eqref{eq:slogs}_2$ and $\eqref{eq:slogs1}$, we can obtain for a.e. $\tau \in (0,T]$:
\be\label{int-logvr}
\int_\O \left( \overline{\vr\log \vr} -  \vr\log \vr \right)(\tau,\cdot)\,\dx + \int_0^\tau \int_\O\left( \overline{\vr \Div  \vu} - {\vr \Div  \vu}  \right)\,\dx\,\dt =0.
\ee

For any $\psi(x)\in C_c^\infty(\O)$ and any $\phi(t)\in C_c^\infty(0,\tau)$ with $\tau \in (0,T]$, there holds
\ba\label{flux1}
\lim_{\e\to 0}\int_0^\tau \intO{\phi \psi \left(p(\tilde\vr_\e)-(\frac{4\mu}{3}+\eta)\Div  \tilde\vu_\e\right)\tilde\vr_\e}\,\dt\\
= \int_0^\tau \intO{\phi \psi \left( \overline{p(\vr)\vr}-(\frac{4\mu}{3}+\eta) \overline{ \vr \Div  \vu}\right)} \,\dt.\nn
\ea
This gives, by using Lemma \ref{lem:flux},
\ba\label{flux2}
\int_0^\tau \int_\O \phi \psi \left( \overline{\vr \Div  \vu} - {\vr \Div  \vu}  \right)\,\dx\,\dt = (\frac{4\mu}{3}+\eta)^{-1} \int_0^\tau \int_\O \phi \psi \left( \overline{p(\vr)\vr } - \overline{p(\vr)} \vr  \right)\,\dx\,\dt.\nn
\ea
This implies, by choosing test function sequences that converges to $1$ strongly in some proper spaces,
\ba\label{flux3}
\int_0^\tau \int_\O  \left( \overline{\vr \Div  \vu} - {\vr \Div  \vu}  \right)\,\dx\,\dt = (\frac{4\mu}{3}+\eta)^{-1} \int_0^\tau \int_\O \left( \overline{p(\vr)\vr } - \overline{p(\vr)} \vr  \right)\,\dx\,\dt.\nn
\ea
Together with \eqref{int-logvr}, we obtain
\ba\label{flux4}
\int_\O \left( \overline{\vr\log \vr} -  \vr\log \vr \right)  (\tau,\cdot) \,\dx + (\frac{4\mu}{3}+\eta)^{-1} \int_0^\tau \int_\O \left( \overline{p(\vr)\vr } - \overline{p(\vr)} \vr  \right)\,\dx\,\dt = 0.
\ea

Since the function $s\to s\log s$ is convex in $[0,\infty)$, there holds,
\be\label{positivity1}
\overline{\vr\log \vr} \geq \vr\log \vr, \quad \mbox{a.e. in $(0,T)\times \O$}.\nn
\ee
Since $p(s)$ is strictly increasing in $[0,\infty)$, there holds (see Theorem 10.19 in \cite{F-N-book}):
\be\label{positivity2}
\overline{p(\vr)\vr} \geq \overline{p(\vr)}\vr,\quad \mbox{a.e. in $(0,T)\times \O$}.\nn
\ee
Thus, by \eqref{flux4}, necessarily, there holds
\be\label{con-st-vr0}
\overline{\vr\log \vr} = \vr\log \vr, \quad \overline{p(\vr)\vr} = \overline{p(\vr)}\vr,\quad \mbox{a.e. in $(0,T)\times \O$},\nn
\ee
which implies the strong convergence
\be\label{con-st-vr}
\tilde \vr_\e \to \vr \quad \mbox{a.e. in $(0,T)\times \O$}.\nn
\ee
This implies the convergence of the pressure term as in \eqref{weak-limit-pressure}.

\subsection{End of the proof}

By the convergence results we obtained in Section \ref{sec:limitpass},  we have shown that the limit $(\vr,\vu)$ is a weak solution to the compressible Navier--Stokes equations in the sense of Definition \ref{def-weaksol} in homogeneous domain $\O$, if the energy inequality is satisfied. Indeed, by the strong convergence of the density shown in \eqref{con-st-vr}, together with \eqref{wl-pd1} and \eqref{wl-pd3}, the energy inequality can be obtained directly by passing $\e\to 0$ in the energy equality for $(\vr_\e, \vu_\e)$. This completes the proof of Theorem \ref{Tm1}.

\subsection*{Acknowledgments}
The authors wish to thank E. Feireisl for inspiring discussions related to the work. S. Schwarzacher gratefully acknowledges the support of the project LL1202 financed by the Ministry of Education, Youth and Sports.


\end{document}